\documentclass{amsart}
\usepackage[utf8]{inputenc}
\usepackage{xcolor}
\usepackage{hyperref}
\usepackage{tikz-cd}
\usepackage{amssymb}
\usepackage{mathabx}
\usepackage{soul}

\newtheorem{thm}{Theorem}[section]
\newtheorem{cor}[thm]{Corollary}
\newtheorem{lem}[thm]{Lemma}
\newtheorem{remark}[thm]{Remark}
\newtheorem{prop}[thm]{Proposition}
\newtheorem{defin}[thm]{Definition}

\newcommand {\ignore}[1]  {}
\newcommand{\df}{{\, \stackrel{\mathrm{def}}{=}\, }}
\newcommand{\sm}{\smallsetminus}
\newcommand{\hol}{\mathrm{hol}}
\newcommand{\M}{\mathcal{M}}
\newcommand{\C}{\mathbb{C}}

\newcommand{\Res}{{\mathrm{Res}}}
\newcommand{\g}{\mathfrak{g}}

\newcommand{\SL}{\operatorname{SL}}

\newcommand{\N}{\mathcal{N}}
\renewcommand{\H}{\mathcal{H}}

\newcommand{\Hm}{\mathcal{H}_{\mathrm{m}}}

\newcommand{\R}{\mathbb{R}}

\newcommand{\dev}{\mathrm{dev}}
\newcommand{\marked}{\H_{\mathrm{m}}}
\renewcommand{\v}{\vec{v}}
\newcommand{\Rel}{\mathrm{Rel}}

\renewcommand{\R}{\mathbb{R}}
\newcommand{\B}{B^{\perp}_{\N}}
\newcommand{\Mod}{\mathrm{Mod}}
\renewcommand{\SS}{\mathfrak{S}}

\newif\ifdraft\drafttrue
\title[Horocycle dynamics in rank one invariant subvarieties I]{Horocycle
 dynamics in rank one invariant subvarieties I: weak measure
  classification and equidistribution}
\author{Jon Chaika}
\address{University of Utah {\tt chaika@math.utah.edu }}

\author{Barak Weiss}
\address{Dept. of Mathematics, Tel Aviv University, Tel Aviv, Israel
{\tt barakw@post.tau.ac.il}}
\author{Florent Ygouf}
\address{Dept. of Mathematics, Tel Aviv University, Tel Aviv, Israel
{\tt florentygouf@mail.tau.ac.il}}

\begin{document}

\begin{abstract}
Let $\M$ be an invariant subvariety in the moduli space of translation surfaces. We contribute to the study of the dynamical properties of the horocycle flow on $\M$. In the context of dynamics on the moduli space of translation surfaces, we introduce the notion of a `weak classification of horocycle invariant measures' and we study its consequences. Among them, we prove genericity of orbits and related uniform equidistribution results, asymptotic equidistribution of sequences of pushed measures, and counting of saddle connection holonomies. As an example, we show that invariant varieties of rank one, Rel-dimension one and related spaces obtained by adding marked points satisfy the `weak classification of horocycle invariant measures'. Our results extend prior results obtained by Eskin-Masur-Schmoll, Eskin-Marklof-Morris, and Bainbridge-Smillie-Weiss. 
\end{abstract}

\maketitle
\tableofcontents

\section{Introduction} 
Any stratum of translation surfaces $\H$ is endowed with an action of the group $G \df \SL_2(\R)$, and the restriction of this action to the subgroup 
\begin{equation}\label{eq: besides}
    U = \{u_s: s \in \R\}, \ \ \ \text{ where } \ u_s \df \left( \begin{matrix} 1 & s \\ 0 & 1\end{matrix} \right),
\end{equation}
\noindent is called the {\em horocycle flow}. It is a longstanding open problem to understand the invariant measures for the $U$-action on $\H$. It has been known for a while, that a `sufficiently nice' classification of such measures would have many interesting dynamical and geometrical consequences. See \cite{Zorich_survey, HMSZ_problems}, and see \cite{chaika2020tremors} for a survey of recent results. This paper contributes to the investigation of this problem in two ways. Firstly, we obtain a  partial classification of horocycle invariant ergodic measures, which we call {\em weak classification}, in the setting of {\em rank-one invariant subvarieties of Rel-dimension one}, or spaces obtained from them by adding marked points. All of these notions will be defined in \S \ref{section: prelims}. Secondly, we explore the implications of a weak classification result in rank-one invariant subvarieties (of arbitrary Rel-dimension) and show that it implies genericity of orbits, equidistribution of circle averages, and many additional equidistribution results. 

We do not know whether the weak classification result holds in general rank-one invariant subvarieties; our approach crucially uses the Rel-dimension one assumption --- see Remark \ref{remark: where is this used}. Nevertheless, in a follow-up paper \cite{CWY2}, we extend the classification of $U$-orbit closures to general rank-one invariant subvarietes (omitting the assumption of one-dimensional Rel). We do so without obtaining a weak measure classification, but rather by an induction on the dimension of Rel. Although the current paper uses completely different techniques, it serves as the base of the induction used in \cite{CWY2}. 

In the remainder of this introduction we give an initial description of the main results and discuss the proofs. Precise statements of results require some preparations and will be given in the body of the text. Let $\M$ be an orbit-closure for the action of $G$. As was shown in a celebrated paper of Eskin, Mirzakhani and Mohammadi \cite{eskin2015isolation}, such orbit-closures are (iso-area subloci of) affine orbifolds in period coordinates, and carry a wealth of additional structure. Following \cite{ApisaWright}, we refer to these orbit closures as {\em invariant subvarieties}. A fundamental invariant of an invariant subvariety is its {\em rank}, introduced in \cite{Wright_cylinders}. We will be interested in the rank-one case, which is characterized by the fact that the $G$-action and the Rel foliation (see \S \ref{subsec: rel foliation}) are complementary foliations, in the sense that a neighborhood of a point can be described in terms of the $G$-action and the Rel-foliation. A further invariant is the {\em Rel dimension} (the complex dimension of leaves of the Rel foliation). Now suppose $\N \subset \M$ is an invariant subvariety. It was shown in \cite{eskin2018invariant} that $\N$ is the support of a unique $\SL_2(\R)$-invariant Radon probability measure, which has a geometric description in terms of the affine  structure of $\N$. We denote this measure by $m_\N$. It is well-known that $m_{\N}$ is ergodic for the action of $U$ and, as first observed by Calta \cite{Calta} and explained in \S \ref{push}, such a measure can be pushed by real Rel to create more $U$-invariant ergodic measures. We say that $\M$ satisfies the {\em weak classification of $U$-invariant measures} if any ergodic $U$-invariant probability Radon measure measure on $\M$, that assigns measure zero to  the set of surfaces with horizontal saddle connections, is obtained by this construction. In Theorem \ref{classification} we prove that  invariant subvarieties of rank one and with Rel-dimension one satisfy weak classification. We remark that all invariant subvarieties modeled on a complex vector space of dimension 3 are of rank one with Rel-dimension one and thus satisfy the hypotheses of Theorem \ref{classification}; explicit examples are given in Corollaries \ref{cor: known examples}, \ref{cor: more known examples} and \ref{cor: more more known examples}. 

Theorem \ref{preequidistribution} gives severe restrictions on the collection of weak-* limits of measures obtained by averaging along a $U$-orbit in a rank-one invariant subvariety. This theorem uses the rank-one structure and is false in general, as shown in \cite{chaika2020tremors}. 
We then explore the implications of weak classification in the setting of rank-one invariant subvarieties (omitting the assumption that the Rel-dimension is one). Under these conditions,  we can derive the following results. In Theorem
\ref{genericity} we prove that the $U$-orbit of any surface without horizontal saddle connections is generic for some $U$-ergodic invariant measure. Moreover, for a given test function $f$ and a given error $\varepsilon$, genericity holds uniformly outside certain compact `singular sets', see Theorem \ref{uniform}. This result is an analogue of \cite[Thm. 3]{DaniMargulis}. We show in Theorem \ref{geodesicpush} that for any $U$-invariant ergodic measure $\mu$, the pushforwards  $g_{t\ast} \mu$ converge to $m_\N$ for some invariant subvariety $\N$, as $t \to \infty$ and to $m_{\N'}$ for some invariant subvariety $\N' \subset \N$, as $t \to -\infty$ (unless $\mu$ assigns positive measure to the set of surfaces with horizontal saddle connections, in which case they diverge in the space of probability measures). Here 
\begin{align}\label{eq: geodesic flow}
    g_t \df \mathrm{diag}(e^t, e^{-t})
\end{align}
\noindent denotes the geodesic flow. In Theorem \ref{pushbyrel} we obtain a similar result for pushing a $G$-invariant measure by a real Rel flow. Finally we show in Theorem \ref{thm: circle averages} that in a rank one locus satisfying the weak classification of $U$-invariant measures, pushed circle averages $\frac{1}{2\pi} \int_0^{2\pi} g_{t \ast} r_\theta \delta_q $ converge as $t \to \infty$, to the $G$-invariant measure on $\N = \overline{Gq}$. Here $r_\theta$ is the rotation matrix by angle $\theta$. As a consequence, every surface satisfies quadratic growth of saddle connection holonomies. Our technique revisits and adapts arguments of Ratner \cite{Ratner} for proving the measure classification of unipotent flows on homogeneous spaces, and subsequent work of Dani and Margulis \cite{DaniMargulis} for harnessing this in order to obtain equidistribution results. The first papers in which these ideas were adapted to translation surfaces are \cite{EMSch, EMM1}, in which a weak classification of $U$-invariant measures was proved, and equidistribution for circle averages was deduced, in the context of branched covers of Veech surfaces. 
%Our Corollary \ref{cor: more more more %known examples} extends this %\textcolor{red}{by allowing to substitute %the Veech surface by a surface whose orbit %closure is a rank-one, Rel-one invariant %subvariety}. 
In \cite{CaltaWortman} and \cite{bainbridge2016horocycle}, measure classification results were obtained in eigenform loci in genus two, which are particular cases of a rank-one invariant subvariety with Rel-dimension one. Specifically, the paper \cite{bainbridge2016horocycle} proves a complete measure classification result, which is stronger than our `weak classification', as it also classifies measures for which typical surfaces do have horizontal saddle connections. From this strong measure classification, equidistribution results are deduced, namely a strong version of Theorem \ref{genericity}, and analogues of Theorems \ref{geodesicpush}, \ref{pushbyrel} and \ref{thm: circle averages}. Our work adapts the techniques in \cite{bainbridge2016horocycle} to our more general setting  and establishes that most of these equidistribution results, and additional ones, can be proved assuming a weaker measure classification result. A key technical input is Proposition \ref{frequencyofpassage}, which improves \cite[Proof of Thm. 11.1, Step 3]{bainbridge2016horocycle} and uniformly bounds the time horocycle trajectories starting in a given compact set can spend close to certain singular sets. 

\subsection{Acknowledgements} The authors gratefully acknowledge
the support of grants BSF 2016256,  ISF 2019/19, ISF-NSFC 3739/21,  a Warnock chair, a Simon Fellowship, 
and NSF grants DMS-1452762 and DMS-2055354.  

\section{Strata of translation surfaces}\label{section: prelims}
In this section we collect some standard definitions and set our
notations. See \cite{MT, Zorich_survey, 
  Forni_Matheus_survey, Wright_survey, bainbridge2016horocycle} for
helpful surveys and additional 
introductory material.

\subsection{Translation surfaces and developing map}
Let $S$ be a compact and oriented topological surface of genus $g$,
let $\Sigma$ be a finite subset of $S$ and let $\kappa = (k_\xi)_{\xi \in \Sigma}$ be an integer partition of $2g-2$. That is, $\sum_{\xi \in \Sigma} k_\xi = 2g-2$ and $k_\xi$ are non-negative integers. We emphasize that $\kappa_\xi = 0$ is allowed, in which case $\xi$ is called a {\em marked point}. A {\em singularity-labeled translation surface of type $\kappa$} is a triple $q = (X_q,\omega_q,\mathfrak{z}_q)$ where $X_q$ is a compact and connected Riemann surface of genus $g$, $\omega_q$ is a non-zero holomorphic $1$-form on $X_q$ whose set of zeroes is denoted by $\Sigma_q$ and $\mathfrak{z}_q : \Sigma \to X_q$ is an injective map whose image contains $\Sigma_q$ and such that $\mathfrak{z}_q(\xi)$ has order $k_\xi$. Where it will not cause any confusion we will omit the words `singularity-labeled' and refer to $q$ simply as a {\em translation surface of type $\kappa$}.  A {\em
  marking} on a translation surface $q$ is a homeomorphism $\varphi: S \to X_q$ such that $\varphi(\xi) = \mathfrak{z}_q(\xi)$ for $\xi \in \Sigma$, and a pair $(q, \varphi)$
is called a {\em marked translation surface with labeled singularities}. Two marked translation surfaces $(q_1,\varphi_1)$ and $(q_2,\varphi_2)$ are isomorphic if there is a biholomorphism $\phi : X_{q_1} \to X_{q_2}$
such that $\phi^{\ast} \omega_{q_2} = \omega_{q_1}$, $\phi \circ
\mathfrak{z}_{q_1} = \mathfrak{z}_{q_2}$ and $\varphi_2^{-1} \circ
\phi \circ \varphi_{1}$ is isotopic to the identity rel $\Sigma$. Let
$\H_{\mathrm{m}}$ be the space of isomorphism classes of marked translation surfaces of type $\kappa$ with labeled singularities and let $\mathrm{Mod}(S,\Sigma)$
be the mapping class group of $S$ that preserves $\Sigma$
point-wise. It acts by precomposition on the space $\H_{\mathrm{m}}$ and we
denote by $\H$ the quotient and by $\pi: \H_{\mathrm{m}} \to \H$ the
quotient map. The space $\H = \H(\kappa)$ is the
{\em moduli space of translation surfaces of type $\kappa$ with
  labeled singularities}. There is a structure of linear
manifold on $\H_{\mathrm{m}}$ that turns the following map into a
local linear isomorphism. 
$$
\dev: 
%\begin{matrix}
\H_{\mathrm{m}}  \to  H^1(S,\Sigma;\mathbb{R}^2), \ \ \
\dev\left((q,\varphi)  \right) \df  \left(\gamma \mapsto
  \int_{\varphi(\gamma)} \omega_q \right).
%\end{matrix}
$$

Each homeomorphism of $S$ preserving $\Sigma$ acts on homology and
this induces an action of 
the group $\Mod(S,\Sigma)$ on $H^1(S,\Sigma;\mathbb{R}^2)$. 
%\combarak{This used to be called `the Torelli representation' but I
%  did not see this in the literature, Farb-Margalit call it `the
%  symplectic representation'}
For this action, the map $\dev$ is
$\Mod(S,\Sigma)$-equivariant. The space $\H$ can be endowed
with the quotient topology and inherits an orbifold structure from the
one on $\H_{\mathrm{m}}$.

The {\em area} of a translation surface is the integral of the volume
form induced by the 1-form, i.e.
$$
\mathrm{Area}(q) = \frac{1}{2\mathbf{i}} \int_S \omega_q \wedge
\overline{\omega}_q. 
$$
We denote the subset of area-one surfaces in $\H_{\mathrm{m}}$ and
$\H$ by $\H_{\mathrm{m}}^{(1)}$ and $\H^{(1)}$ respectively.

Associated with a translation surface is a translation atlas on $X_q$,
obtained by locally integrating $\omega_q$. 
A {\em saddle connection} on a translation surface $q$ is a path
$\gamma: [0,1] \to X_q$ such that $\gamma(0), \gamma(1) \in \Sigma_q,
\gamma(t) \notin \Sigma_q$ for $t \in (0,1)$, and $t \mapsto
\gamma(t) $ is linear in planar charts. We refer to the vector
$\int_\gamma \omega_q \in \C \cong \R^2$ as the {\em holonomy} of $\gamma$
with respect to $q$, and denote it by $\hol_q(\gamma)$. 

\subsection{$G$-action, invariant subvarieties}
For the remainder of this paper we set $G \df \SL_2(\R)$. 
The space $\H_{\mathrm{m}}$ is endowed with a $G$-action, for which the developing map is equivariant with respect to the natural 
action of $G$ by post-composition on $H^{1}(S,\Sigma;\R^2)$. More precisely, for any  $(q,\varphi)$ and any $g \in G$, we have

\begin{equation*}
    \dev( g (q,\varphi) ) = g \dev(q,\varphi),
\end{equation*}
and thus, for any saddle connection $\delta$ on $q$ and any $g \in G$,
\begin{equation}\label{eq: saddle connection equivariance}
    g \hol_{q}(\delta) = \hol_{gq}(\delta).
\end{equation}
The $G$-action preserves the area of a translation surface. Besides the group $U = \{u_s : s \in \R\}$ defined in \eqref{eq: besides}, we will consider 
\begin{equation}\label{eq: some matrices}
    A \df \{g_t : t \in \R\} \text{ and } B \df AU,
\end{equation}
where $g_t$ is as in \eqref{eq: geodesic flow}. The action of $A$ will be referred to as the {\em geodesic flow.}

\begin{defin}[Invariant subvariety]\label{def: invariant subvariety}
An invariant subvariety is a subset $\M \subset \H$ such that $\pi^{-1}(\M)$ is a locally finite union of connected equilinear manifolds, each of which projects to $\M$ under $\pi$.
\end{defin}

We recall from \cite{SSWY} that an equilinear manifold is a properly embedded connected submanifold $X \subset \H_{\mathrm{m}}$ that is modeled via the developing map on a $\C$-linear subspace $V \subset H^1(S,\Sigma;\C) \cong H^1(S, \Sigma; \R^2)$ defined over $\R$, and such that if $\gamma \in \Mod(S,\Sigma)$ preserves $X$, then $\gamma|_{V}$ has determinant $1$. By the work of Eskin, Mirzakhani and Mohammadi (\cite{eskin2018invariant} and \cite{eskin2015isolation}), the orbit-closures for the $\mathrm{GL}_2^+(\R)$-action are exactly the invariant subvarieties, closed invariant sets are always finite unions of invariant subvarieties, and the orbit-closures and invariant measures for $G$ coincide with the orbit-closures and invariant measures for its subgroup $P$. What we call `invariant subvariety' was called `affine invariant manifold' in \cite{eskin2015isolation}; our choice follows \cite{ApisaWright}. Note also that Definition \ref{def: invariant subvariety} differs slightly from the definitions given in \cite{eskin2015isolation} but will be more convenient for us, see \cite{SSWY} for a more detailed discussion of this definition. 

Let $\M$ be an invariant subvariety. A connected equilinear manifold $X$ contained in $\pi^{-1}(\M)$ such that $\pi(X) = \M$ is called an {\em irreducible component of $\pi^{-1}(\M)$}. It is easy to see that $\pi^{-1}(\M)$ is the union of its irreducible components and that $\Mod(S,\Sigma)$ acts transitively on the set of irreducible components. In particular, since the action of $\Mod(S,\Sigma)$ is smooth, the irreducible components all have the same dimension. This allows us to define the dimension of $\M$, which we denote by $\dim(\M)$, as the dimension of any one of the irreducible components. 

Let $\M^{(1)}$ be the area-one locus in $\M$. Using Lebesgue measures on the irreducible components of $\pi^{-1}(\M)$ and a `cone construction', one can construct a natural smooth $G$-invariant Radon probability measure of full support on $\M^{(1)}$. For more details on this construction, using the language above, see \cite{SSWY}. We will denote this measure by $m_\M$. It was shown in \cite{eskin2018invariant} that all $G$-invariant ergodic measures on strata arise in this way, and that $m_\M$ is the unique $G$-invariant ergodic measure with $\mathrm{supp} \, m_\M = \M^{(1)}$. 

\subsection{Rel foliation, Rel vector fields} \label{subsec: rel foliation}
We recall the following construction (see \cite{Zorich_survey, MW2, McMullen_isoperiodic, bainbridge2016horocycle}). Let $\Res : H^1(S,\Sigma;\R^2) \to H^1(S; \R^2)$ be the canonical cocycle restriction map and denote its kernel by $\mathfrak{R}$. Clearly $\mathfrak{R}$ is a $\C$-linear subspace defined over $\R$. It will be useful to have an explicit set of generators for $\mathfrak{R}.$
For any 
$\xi \in \Sigma$, denote by $\xi^*$ the element of $H^1(S,\Sigma;\R^2)$ that is dual (with respect to the intersection form) to the cycle of $H_1(S \sm \Sigma;\mathbb{Z})$ represented by a contractible loop on $S$ that circles counterclockwise the point $\xi$. The elements $\{\xi^* : \xi \in \Sigma\}$ are a set of generators for $\mathfrak{R}$, and they satisfy one relation $\sum_{\xi \in \Sigma} \xi^* =0.$ In particular $\dim_{\C}(\mathfrak{R}) = n-1$, where $n$ is the cardinality of $\Sigma$.  

There is a foliation of $H^1(S,\Sigma;\R^2)$ by translates of $\mathfrak{R}$ which can be pulled back by the developing map to give a foliation of $\H_{\mathrm{m}}$. Since the map $\dev \circ \Res$ is $\Mod(S,\Sigma)$-equivariant, there is an induced foliation of $\H$. The area of two translation surfaces which are in the same leaf for this foliation are the same and thus the foliation restricts to give a foliation of $\H^{(1)}$. We call this foliation (on any one of the spaces $\H_{\mathrm{m}}, \H,$ or $\H^{(1)}$) the {\em Rel foliation}; it has also been called the kernel foliation, the isoperiodic foliation, and the absolute period foliation. 

Recall our convention that the singularities of surfaces in $\H$ are labeled. This implies that elements of 
$\Mod(S,\Sigma)$ do not permute the singularities and act trivially on $\mathfrak{R}.$
For any $v \in \mathfrak{R}$, the constant vector field  $v$ can be lifted to a vector field on $\marked$ by $\dev$ and since singularities are not permuted by $\Mod(S, \Sigma)$, this vector field descends to a well-defined vector field on $\H$, which is called the Rel vector field associated to $v$ and denoted $\v$.

\begin{defin}
Let $q \in \H$ and let $v \in \mathfrak{R}$. We say that $\Rel_v(q)$ is well-defined and equal to $q'$ if there is a smooth path $\gamma: I \to \H$, where $I$ is an interval that contains $0$ and $1$, and such that $\gamma(0) = q$, $\gamma(1)=q'$ and $\forall t \in I, \ \frac{d}{dt}\gamma(t) = \v$. We denote by  $\Delta$ the subset of $ \H \times \mathfrak{R}$ consisting of pairs $(q,v)$ for which $\Rel_v(q)$ is well-defined.  
\end{defin}

We summarize some properties of this partially defined map, see \cite[\S4]{bainbridge2016horocycle} for proofs.

The set $\Delta$ is open and the map
\begin{align*}
    \Delta \to\H,  \ \ \ \ \ (q,v) \mapsto \Rel_v(q)
\end{align*}
is continuous. The group $G$ acts on $\mathfrak{R}$ by restriction of its action by post-composition on $H^1(S,\Sigma;\C)$, and for any $g \in G$, 
\begin{equation}\label{eq: equivariance Rel}
    (q, v) \in \Delta \iff (gq, gv) \in \Delta, \ \ \text{ and } \ 
    \Rel_{gv}(gq) = g \Rel_v(q).
\end{equation}
The set $\Delta$ is completely described in \cite{MW2, bainbridge2016horocycle} in terms of saddle connections. In particular, if $n=2$ and $q$ has no saddle connection whose holonomy is parallel to $v$, then $(q,v) \in \Delta$.

Let $\M \subset \H$ be an invariant subvariety, let $X$ be an irreducible component of $\pi^{-1}(\M)$ and let $V$ be the $\C$-linear subspace of $H^1(S,\Sigma;\C)$ defined over $\R$, on which $X$ is modeled. We define 
$$
\mathfrak{R}_X \df \mathfrak{R} \cap V.
$$
Since $\Mod(S,\Sigma)$ acts transitively on the set of irreducible components of $\pi^{-1}(\M)$ and $\Mod(S,\Sigma)$ acts trivially on $\mathfrak{R}$, it is easy to see that if $Y$ is another irreducible component then $\mathfrak{R}_X = \mathfrak{R}_Y$. We thus denote by $\mathfrak{R}_{\M}$ this common subspace. By construction, for any $q \in \M$ and $v \in \mathfrak{R}_{\M}$ we have $\Rel_v(q) \in \M$ whenever the latter is well-defined, and $\mathfrak{R}_{\M}$ is the biggest linear subspace of $\mathfrak{R}$ with this property. Whenever $\N \subset \M$ is an inclusion of invariant subvarieties, we have $\mathfrak{R}_\N \subset \mathfrak{R}_\M.$
 We refer to the number $\dim_{\C}(\mathfrak{R}_\M)$ as the {\em Rel dimension of $\M$.} 

Let $V \subset H^1(S, \Sigma; \C)$ be the local model of an irreducible component $X$ of $\pi^{-1}(\M)$. It follows from \cite[Theorem 1.4]{avila2017symplectic} that $\Res(V)$ is a symplectic subspace of $H^1(S; \C)$ with respect to the symplectic structure induced by the cup product, and in particular, that $\dim_{\C} (\Res(V))$ is even. In particular this implies that for any affine subvariety $\M$, $\dim_{\C}(\M) - \dim_{\C}(\mathfrak{R}_{\M})$ is an even number.

\begin{defin}[Rank]\label{def: rank}
The {\em rank} of $\M$ is defined to be 
$\frac12 \left(\dim_{\C}(\M) - \dim_{\C}(\mathfrak{R}_{\M})\right)$. 
\end{defin}

The rank is a fundamental invariant of an invariant subvariety. See \cite{Wright_cylinders} for additional characterizations. In this paper we will be particularly interested in the case $\mathrm{rank}(\M)=1$. A simple dimension computation implies that in this case, the Rel leaves and the foliation into $G$-orbits are complementary foliations on $\M^{(1)}$; that is, in the rank one case, for any $q \in \M$ there are neighborhoods $\mathcal{U}$ in $\M^{(1)}$ and neighborhoods $\mathcal{V}_1, \mathcal{V}_2$ of the identity in $G$ and of $0$ in $\mathfrak{R}$, such that the map 
$$
\mathcal{V}_1 \times \mathcal{V}_2 \to \mathcal{U},  \ \ \ \ \ (g,v)
\mapsto \Rel_v(gq)
$$
\noindent is a homeomorphism. It follows from the discussion preceding Definition \ref{def: rank} that if $\dim_{\C}(\M)=2$ then necessarily $\M^{(1)}$ is a closed $G$-orbit, and that being a rank-one invariant subvariety of Rel-dimension one is equivalent to the condition $\dim_{\C}(\M)=3$.

For fixed $v \in \mathfrak{R}$, we denote by $\H'_v$ the set of surfaces $q \in \H$ such that $(q,v) \in \Delta$; and for any $v \in \mathfrak{R}_{\M}$, we denote by $\M'_v$ the set of surfaces $q \in \M$ such that $(q,v) \in \Delta$.

\subsection{Real Rel vector fields and real Rel flow}
Following \cite{bainbridge2016horocycle}, we write $\R^2 = \R_x \oplus \R_y$ for the canonical splitting into $x$ and $y$ coordinates. This induces the splitting
\begin{align}\label{eq: splitting}
    H^1(S,\Sigma;\R^2) = H^1(S,\Sigma;\R_x) \oplus H^1(S,\Sigma;\R_y),
\end{align}
\noindent where $H^1(S, \Sigma; \R_x) \cong H^1(S, \Sigma; \R) \cong H^1(S, \Sigma; \R_y)$.  The subspace $H^1(S,\Sigma;\R_x)$ shall be referred to as the \textit{horizontal space}. If $v \in H^1(S,\Sigma;\R_x)$, it will be convenient to denote by $\mathbf{i} \cdot v$ the corresponding vector in $H^1(S,\Sigma;\R_y)$; that is, with respect to the splitting \eqref{eq: splitting}, $(x,0)$ projects to $x \in H^1(S,\Sigma;\R_x)$ and $(0, x)$ projects to $\mathbf{i} \cdot x \in H^1(S,\Sigma;\R_y)$. We denote
$$
Z \df \mathfrak{R} \cap
H^1(S,\Sigma;\R_x) \ \ \ \text{ and } Z_{\M} \df \mathfrak{R}_{\M}
\cap  H^1(S,\Sigma;\R_x) 
$$
\noindent where  $\M \subset \H$ is an invariant subvariety. If $v \in Z$, the vector field $\v$ is called a {\em real Rel vector field}. For any $r>0$, we denote by $\M_r$ the set of surfaces in $\M$ which do not have a horizontal saddle connection strictly shorter than $r$, and by $\M_{\infty}$ the set of surfaces in $\M$ that  do not have horizontal saddle connections. With respect to the  action of $G$ by post-composition on $\mathfrak{R}$, the space  $\mathfrak{R}_{\M}$ is invariant, and the subgroup $B \subset G$ defined in \eqref{eq: some matrices} leaves the subspace $Z_{\M}$ invariant. We can use this action to define semi-direct products, as follows

\begin{defin}

Let $\M \subset \H$ be an invariant subvariety. We define
\begin{align*}
    N_{\M} \df B \ltimes Z_{\M} \ \text{ and } \ L_{\M} \df G \ltimes
  \mathfrak{R}_{\M}. 
\end{align*}

\end{defin}

The set $\M_\infty$ is a dense $G_\delta$ subset of $\M$. Since $B$ preserves horizontal saddle connections and real Rel is defined on surfaces without horizontal saddle connections, the map
$$
N_\M \times \M_\infty \ni (b,z,q) \mapsto \Rel_z(bq)
$$

\noindent is well-defined. Moreover, it satisfies the axioms of a group action (see \cite{MW2}). In particular, for $z \in Z_\M$, the map $(t,q) \mapsto \Rel_{tz}(q)$ defines a {\em real Rel flow} on $\M_\infty$, and this flow commutes with the horocycle flow. On the other hand, the same map defined on the larger set $L_\M \times \M$ need not  satisfy the axioms of a group action, even if well-defined. In order to stress this point, for general $q \in \M$ and $(g,v) \in L_{\M}$ for which $\Rel_{v}(g q)$ is well-defined,  we will write $g q \pluscirc v \df \Rel_{v}(g q).$ 

\subsection{Real Rel pushes of ergodic $U$-invariant measures}\label{push}

For fixed $q \in \H$, let $Z^{(q)}$ denote the set of $z \in Z$ for which $\Rel_z(q)$ is well-defined. By \eqref{eq: equivariance Rel}, $Z^{(q)} = Z^{(uq)}$ for any  $u \in U$ and $q \in \H$, and hence, for any ergodic $U$-invariant probability measure $\mu$ on $\H$,  there is a set $Z^{(\mu)} \subset Z$ such that  $Z^{(q)} = Z^{(\mu)}$ for $\mu$-a.e. $q \in \H$. Let $z \in Z^{(\mu)}$. Denoting by $\mathcal{B}(\M)$ the Borel $\sigma$-algebra on $\M$, we define a measure by
\begin{align}
    \forall X \in \mathcal{B}(\M), \ \ \ {\Rel_z}_{\ast}\mu(X) = \mu(\Rel_{-z}(\H_{-z}' \cap X)). 
\end{align}

The measure ${\Rel_{z}}_{\ast}\mu$ is also $U$-invariant and ergodic. In particular, let $\M \subset \H$ be an invariant subvariety. If $\mu$ is an ergodic $U$-invariant measure supported on $\M$ and $z \in Z^{(\mu)} \cap Z_{\M}$, then ${\Rel_{z}}_{\ast}\mu$ is also an ergodic $U$-invariant measure supported on $\M$. We finally denote
\begin{align}\label{eq: definition of the normalizer}
    N_{\mu} \df \left\{(b,z) \in N_\M :  z \in Z^{(b_{\ast}\mu)} \ \text{and} \ {\Rel_z}_{\ast}( b_{\ast}\mu) = \mu \right\}. 
\end{align}

It follows from \cite[Cor. 4.20]{bainbridge2016horocycle} that $N_{\mu}$ is a closed subgroup of $N_\M$, and thus $N^\circ_\mu$, the connected component of the identity in $N_{\mu}$, is a connected Lie group. We recall that a point $q  \in \H$ is generic for an ergodic $U$-invariant measure $\mu$ if for any compactly supported continuous function $\varphi : \M \to \R$, 
\begin{align*}
    \frac{1}{T} \int_0^T f(u_s  q) \, ds \underset{T \to \infty}{\longrightarrow} \int_{\H} f \ d\mu. 
\end{align*}

The set of generic points of $\mu$, which we denote by $\Omega_{\mu}$, has full measure for $\mu$. We will need the following: 

\begin{prop}[\cite{bainbridge2016horocycle}, Cor. 4.24, Prop. 4.27 \& Cor. 4.28]\label{prop: first}  
Let $\mu$ be an ergodic $U$-invariant measure.
Then:
\begin{itemize}
\item$Z^{(\mu)} \subset Z^{(q)}$ for any $q \in \Omega_{\mu}$. 
\item
If $q,q' \in \Omega_{\mu}$ and $n = (b,z) \in N_\M$ satisfy that $z \in Z^{(b_{\ast}\mu)}$ and $q' = \Rel_z(bq)$, then $n \in N_{\mu}$.  
\item
If $n =(b,z) \in N_\M$ with $z \in Z^{(\mu)}$, and $q \in \Omega_\mu$, then $\Rel_z(bq) \in \Omega_{\mu'}$ for $\mu' = \Rel_{z\ast} (b_\ast \mu)$.  
\end{itemize}
\end{prop}

We will also need the following result, which is proved in \cite[\S 6]{CaltaWortman} for eigenform loci in genus 2, and is a special case of a general result proved in \cite{SSWY}: 

\begin{prop}\label{prop: from SSWY}
Let $\M$ be a rank-one invariant subvariety and let $\mu$ be a $U$-invariant ergodic probability measure on $\M$. Suppose that $\mu$ assigns zero measure to surfaces with horizontal saddle connections, and that $\Rel_{x\ast} \mu = \mu$ for every  $x \in Z_\M$. Then $\mu = m_\M$.  
\end{prop}

\medskip

{\bf Notational conventions.} We will be interested in the horocycle dynamics on invariant subvarieties of rank one. Since the horocycle flow preserves the area of a surface, it is sufficient to consider the area one locus in an invariant subvariety. {\em From now on, to simplify notation, we will restrict attention to area-one surfaces and we will write $\H$ for $\H^{(1)}$, $\M$ for $\M^{(1)}$, etc.} Also, although some statements are valid in greater generality, {\em from now until the end of the paper, $\M$ will denote a rank one invariant subvariety in $\H$}. 

\section{Horocycle dynamics in rank-one, Rel-one invariant subvarieties and compact extension} 

This section is concerned with the classification of ergodic $U$-invariant measures supported on invariant subvarieties of rank one and Rel-dimension one and related compact extensions where marked points are added. {\em Throughout this section we assume that, in addition to being of rank one, $\M$ is of Rel-dimension one.}

\subsection{The drift argument}\label{subsec: drift}
The following key argument is modeled on Ratner's work and was adapted to the present context in \cite{eskin2006unipotent, bainbridge2016horocycle}. 

\begin{prop}\label{drift}
Let $\mu$ be an ergodic $U$-invariant probability measure on $\M$. If $\mu$ is not supported on a closed $U$-orbit then $\dim(N^\circ_{\mu}) \geq 2$.  
\end{prop}

\begin{remark}
Contrary to \cite[Prop 9.4]{bainbridge2016horocycle}, we do not need to assume that $\mu$ does not assign positive measure to the set of surfaces with horizontal saddle connection.
\end{remark}

\begin{proof}
Let $\varepsilon>0$ and let $K \subset \Omega_{\mu}$ be a compact subset such that $\mu(K)>1-\varepsilon$. Such a compact set exists by inner regularity of the measure $\mu$. By combining Birkhoff's ergodic theorem and Egorov's theorem, one obtains that there is another compact set $K_0 \subset \M$ and $T_0$ large enough so that for any $T>T_0$ and any $q \in K_0$,
\begin{align}
    \frac{1}{T}|\{s \in [0,T] :  u_s q \in K \}| > 1-2\varepsilon.
\end{align}

Let $\g$ denote the Lie algebra of $G$, which we identify with the set of $2 \times 2$ matrices with vanishing trace, endowed with the Lie bracket $[A,B] = AB - BA$, and let $\exp: \g \to G$ be the  exponential map. We also choose a $G$-equivariant linear isomorphism $\mathfrak{R}_{\M} \simeq \R^2$ where the action of $G$ on $\R^2$ is the usual linear action. Under this isomorphism, we have $Z_{\M} \simeq \R \times \{0\}$. In the remainder of this proof, we shall abuse slightly notations and write elements of $\mathfrak{R}_{\M}$ as elements of $\R^2$.

If $\mu$ is not supported on a closed  orbit,
there is a point $q \in K_0$ and a sequence of points $q_n \in K_0 \sm U q$ that converges to $q$. For $n$ large enough, we can find $V_n \in \g$ and $W_n \in \mathfrak{R}_{\M}$ such that 
$$(\exp(V_n),W_n) \in  L_\M \sm U, \ \ \ (V_n, W_n) \to (0,0) \ \ \ \text{ and } \  q_n = \exp(V_n) q \pluscirc W_n.
$$
Write
\begin{align*}
    V_n = \begin{pmatrix} a_n && u_n \\ v_n && - a_n \end{pmatrix}  \  \ \text{ and } \ W_n = (x_y,y_n).
\end{align*}

Let $P_n(s)= -v_ns^2 + (1-2a_n)s + u_n$. A computation using \eqref{eq: equivariance Rel} shows that for any $s>0$, 
\begin{align*}
    u_s q_n = \exp\left(\begin{smallmatrix} a_n + sv_n && 0 \\ v_n && -(a_n + sv_n) \end{smallmatrix} \right) u_{P_n(s)} q \pluscirc (x_n + sy_n,y_n).
\end{align*}
Let $0<\eta<1$ and define 
\begin{align*}
    I_n \df \{s \in \mathbb{R}  : \max( |a_n + sv_n| , |x_n + sy_n|) < \eta \varepsilon \}. 
\end{align*}

If $n$ is large enough, the set $I_n$ contains $0$ and we denote by $J_n$ the connected component of $I_n$ that contains $0$. For any $s>0$, we have $P_n'(s) = -2(sv_n + a_n) + 1$ and thus if $s \in I_n$ we get $|P_n'(s)| < 1 + 2\eta\varepsilon<1+2\varepsilon$. We want to find $s \in J_n$ such that 
\begin{equation}\label{eq: first requirement}
u_s  q_n \text{  and } u_{P_n(s)} q \text{ are both in }K 
\end{equation}
and 
\begin{equation}\label{eq: second requirement}
\frac{\eta\varepsilon}{4} \leq \max(|a_n + sv_n|,|x_n + sy_n|) \leq \eta \varepsilon.
\end{equation}

Notice that if $n$ is large enough then $S_n \df \sup(J_n)$ is larger than $T_0$ and 
\begin{align*}
    \frac{1}{S_n} \left| \left\{s \in [0,S_n]  : \max( |a_n + sv_n| , |x_n + sy_n|) < \frac{\eta\varepsilon }{ 4 }\right\} \right | < \frac12. 
\end{align*}

Finally, the set of times $s \in [0,S_n]$ for which $u_s q_n$ is in $K$ has measure at least $(1-2\varepsilon)S_n$ and likewise for $u_s q$. Since $P_n'(s)$ is bounded  by $1 + 2\varepsilon$ on $J_n$, the set of times $s \in [0,S_n]$ for which $u_{P_n(s)} q$ belongs to $K$ has measure at least $(1-4\varepsilon)S_n$. Consequently, whenever  $\varepsilon$ is small enough, for all $n$ large enough (depending on $T_0$ and hence on $\varepsilon$), $s=s_n$ satisfies the desired properties \eqref{eq: first requirement} and \eqref{eq: second requirement}. We choose $\eta$ small enough to that $\exp(v) p \pluscirc w$ is well defined for any $v$, $w$ with coefficients smaller that $\eta \varepsilon$ and $p \in K$. Passing to a converging subsequence in $u_{s_n}  q$ and $u_{p_n(s)} q$ and invoking Proposition 4.6 in \cite{bainbridge2016horocycle}, we obtain a pair of generic points $p$ and $p'$ in $K$ such that $p' = \exp(V)  p \pluscirc W$ with
\begin{align*} 
    V= \begin{pmatrix} a && 0 \\ 0 && -a \end{pmatrix} \ \mathrm{and} \ W = (x,0), 
\end{align*}
\noindent and $\eta\varepsilon/4 \leq \max(a,x) \leq \varepsilon$. From Proposition \ref{prop: first} that $\exp(V, W) \in N_{\mu}$. Since $\varepsilon$ was arbitrary, this shows that $N_\mu \sm U$ contains elements arbitrarily close to the identity. Since $N_\mu$ is a Lie subgroup of $N_\N$, this implies $\dim N_\mu^\circ >1$. 
\end{proof}

\subsection{Saddle connection free $U$-invariant ergodic  measures}
The set of surfaces with horizontal saddle connections is invariant under $U$. In particular, if $\mu$ is an ergodic $U$-invariant measure, the measure of this set is either 1 or 0. In the second case, we say $\mu$ 
is a {\it saddle connection free measure.} 

\begin{thm}\label{classification}
Let $\mu$ be a saddle connection free ergodic $U$-invariant probability measure on $\M$. Then there is some invariant subvariety $\N \subset \M$ and some $x \in Z_\M$ such that $\mu = \Rel_{x \ast} m_{\N}$. 
\end{thm}

\begin{proof}
%[Proof of Theorem \ref{classification}]
It is well-known that the normalized length measure on a periodic $U$-orbit is not saddle connection free. Thus, by Proposition \ref{drift} there are $a \geq 0$, $w \in Z_{M}$ and a homomorphism $\rho:\R \to N^\circ_{\mu} \sm U$ such that
\begin{equation}\label{eq: one parameter group rho}
    \left. \frac{d}{dt} \right|_{t=0} \rho(t) = \Bigl( \left(\begin{smallmatrix} a && 0 \\ 0 && -a \end{smallmatrix} \right), w \Bigr), \ \ \text{ where } (a, w) \neq (0,0).
\end{equation}

If $a = 0$, we have that $\rho(t) = (\mathrm{Id},t w)$ and thus $\rho(t)_{\ast} \mu = {\Rel_{t w}}_{\ast}\mu = \mu$. Therefore $\mu$  is invariant under both $U$ and $\Rel_{t  w}$ for all $t$. Since $\M$ is of  Rel-dimension 1, we deduce via Proposition \ref{prop: from SSWY} that $\mu = m_{\M}$. If $a \neq 0$ we have  
\begin{align*}
     \rho(t)=\Bigl( \left(\begin{smallmatrix} e^{at} && 0 \\ 0 && e^{-at} \end{smallmatrix} \right), e^{at} w - w \Bigr)
\end{align*}
\noindent and it follows via \eqref{eq: equivariance Rel} that the measure $\mu_0 \df {\Rel_{-w}}_{\ast} \mu$ is  invariant under the upper triangular subgroup $B$. Here we have used the assumption that $\mu$ assigns zero measure to surfaces without horizontal saddle connections, since it implies $Z_\M \subset Z^{(\mu)}$, that is $\Rel_{-w}(q)$ is well-defined for every $w \in Z_\M$, for $\mu$-a.e. $q$. It follows from   \cite{eskin2018invariant} that there is an invariant subvariety $\N \subset \M$ such that $\mu_0 = m_{\N}$ and thus $\mu = \Rel_{w \ast} m_{\N}$. 
\end{proof}

\begin{remark}\label{remark: where is this used}
The drift argument relies on a fine analysis of the displacement  between two horocycle trajectories of nearby points. In order to remove the assumption that $\M$ is of Rel-dimension one, it is tempting to argue inductively, replacing the horocycle flow by a well-chosen flow in the group $N_{\mu}$ in order to pick up additional invariance of the measure. However, to compute the displacement between two nearby trajectories under this new flow, one would need the `group-action property' 
\eqref{eq: equivariance Rel}, with $G$ replaced by larger groups $G \ltimes \mathfrak{R}_0,$ for $\mathfrak{R}_0 \subset \mathfrak{R}$ (even the weaker property $\Rel_{x_1+x_2}(q) = \Rel_{x_1}(\Rel_{x_2}(q))$ for suitable $x_1, x_2 \in \mathfrak{R}$ might suffice). These relations are not true in general. 

\end{remark}

We now list some examples to which our results apply. For relevant definitions, see \cite{mcmullen2006prym}. Let $D \equiv 0 \mod  4$ or $D \equiv 1 \mod 4$, let $\mathcal{O}_D$ be the quadratic order of discriminant $D$ and let $\Omega E_D$ be the space of Prym eigenforms for real multiplication by $\mathcal{O}_D$. Let $\Omega E_D^{odd}(2,2)$ and $\Omega E_D(2,1,1)$ be the intersection of $\Omega E_D$ with the strata $\mathcal{H}^{odd}(2,2)$ and $\H(2,1,1)$. It follows from \cite{lanneau2018connected} that these spaces are not empty whenever $D \not\equiv 5 \ \mathrm{mod} \ 8$, in which case there are at most two connected components. Any such connected component is a rank 1 invariant subvariety of Rel-dimension 1, see \cite[Lemma 3.1]{lanneau2016complete}. Thus, Theorem \ref{classification} has the following corollary:

\begin{cor}\label{cor: known examples}
Let $D \equiv 0,1 \ \mathrm{mod} \ 4$, $D \not \equiv 5 \ \mathrm{mod} \ 8$ be a positive integer that is not a square and let $\M$ be a connected component of $\Omega E_D^{odd}(2,2)$ or $\Omega E_D(2,1,1)$. Then $\M$ satisfies the weak classification of $U$-invariant measures. 
\end{cor}

Let $\M$ be a connected component of $\Omega E_D$. In the case where $D$ is a square, $\M$ consists of branched covers of flat tori and the fact that $\M$ satisfies the weak classification of $U$-invariant measures follows from the work of Eskin, Markloff and Morris \cite[Theorem 2.1]{EMM1}. When $D$ is not a square, $\M$ is a nonarithmetic invariant subvariety, which means that its field of definition is a non trivial algebraic extension of $\mathbb{Q}$. The examples mentioned in Corollary \ref{cor: known examples} do not arise via a covering construction and thus provide the first non-trivial examples in genus 3 of invariant subvarieties in which a classification of $U$-invariant measures is known. It is worth mentioning that any nonarithmetic rank 1 invariant subvariety of Rel-dimension 1 in $\H^{odd}(2,2)$ arises in this way, see \cite{ygouf2020non}. The other Prym eigenform loci in genus 3 are $\Omega(\kappa) \df \Omega E_D \cap \mathcal{H}(\kappa)$ for $\kappa \in \{4^{odd}, 4^{hyp}, (2,2)^{hyp}, (1,1,1,1) \}$. The first three are a finite union of Teichm{\"u}ller curves (\textit{i.e.,} satisfy $dim(\M) = 2$), see \cite[Prop 2.3]{lanneau2018connected} for a proof in the case $(2,2)^{hyp}$, and the latter has Rel-dimension 2. In genus 4 and 5, no explicit description of the connected components of the Prym eigenform loci is known besides the case $\H(6)$ that is addressed in \cite{lanneau2020weierstrass}. Nevertheless, we can still use Prym eigenforms to produce more examples of rank one invariant subvarieties of Rel-dimension 1. We recall the following that lists all the potential rank one invariant subvarieties of Rel-dimension 1 arising from Prym eigenforms: 

\begin{lem}[{\cite[Lemma 3.1]{lanneau2016complete}}]\label{more examples}
Let $D \equiv 0,1 \ \mathrm{mod}\  4 $, let $g=4,5$, let $\kappa$ be an integer partition of $2g-2$ such that $\Omega E_D(\kappa) \neq \emptyset$ and let $\M$ be a connected component of $\Omega E_D(\kappa)$. Then $\M$ is a rank one invariant subvariety and its Rel-dimension is 1 if, and only if, $\kappa$ belongs to the following list:
$$ (3,3), (2,2,2)^{even}, (1,1,4), (1,1,2,2), (4,4)^{even}.$$
\end{lem}

\begin{prop}\label{prop: not empty}
For any $\kappa$ as in Lemma \ref{more examples}, there is $D \equiv 0,1 \ \mathrm{mod}\  4 $ such that $D$ is not a square and $\Omega E_D(\kappa) \neq \emptyset$.
\end{prop}

\begin{proof}
It is explained in \cite{lanneau2016complete} that each $\kappa$ as in Lemma \ref{more examples} is associated with a rank 2 invariant subvariety denoted by $Prym(\kappa)$, and the latter can easily be verified to be arithmetic. It thus follows from \cite[Theorem 1.7]{thealgebraichull} that this affine variety itself contains a nonarithmetic rank one invariant subvariety $\M$. It is not hard to see that $\M$ must contain a surface $q$ that contains a hyperbolic matrix $A$ in its Veech group. See for instance \cite[Remark 4.2]{wright2014field}. It follows from \cite[Theorem 3.5]{mcmullen2006prym} that $q \in \Omega E_D(\kappa)$ with $\mathbb{Q}(\sqrt{D}) = \mathbb{Q}(tr(A))$. If $tr(A)$ is rational then it follows from \cite[Theorem 9.8]{mcmullen2003teichmuller} that $q$ covers a torus, in contradiction with the fact that $\M$ is nonarithmetic.  
\end{proof}

Using Theorem \ref{classification}, we get:

\begin{cor}\label{cor: more known examples}
Let $\M$ be a connected component of $\Omega E_D(\kappa)$ where $\kappa$ and $D$ are as in Proposition \ref{prop: not empty}. Then $\M$ satisfies the weak classification of $U$-invariant measures. 
\end{cor}

The Pentagon variety described in \cite{eskin2020billiards} yields another infinite family of quadratic rank one invariant subvarieties of Rel-dimension 1 contained in $\mathcal{H}(2,2,1,1)$. It is explained in \cite[Appendix B]{eskin2020billiards} that the intersection of the Pentagon variety with $\H(2,2,1,1)$ is a non-empty arithmetic rank two invariant subvariety of Rel-dimension 1. Using once more \cite[Theorem 1.7]{thealgebraichull}, we deduce the existence of infinitely many quadratic rank one invariant subvarieties of Rel-dimension 1. 

\begin{cor}\label{cor: more more known examples}
Let $\M$ be one of the infinitely many quadratic rank one invariant subvarieties contained in the intersection of the Pentagon variety with $\H(2,2,1,1)$.  Then $\M$ satisfies the weak classification of $U$-invariant measures. 
\end{cor}

\section{Adding marked points}\label{adding marked points}
Let $\Sigma$ be a collection of singularities and let $\kappa  = (k_\xi)_{\xi \in \Sigma}$ be a partition of $2g-2$. We will be interested in the case that some of the $k_\xi$ are zero (and thus correspond to marked points). Let 
\begin{align}\label{sigma0}
    \Sigma_0 \df \{\xi \in \Sigma: k_\sigma \neq 0\}
\end{align}
be the subset of `non-removable' singularities, 
let $\mathfrak{S} \subset \Sigma$ be a subset containing $\Sigma_0$, 
and $\kappa_{\SS} \df (k_{\xi})_{\xi \in \SS}$. By definition of $\Sigma_0$, $\kappa_{\SS}$ is also an integer partition of $2g-2$. Let $\H_\SS \df \H(\kappa_{\SS})$ be the moduli space of singularity-labeled translation surfaces of type $\kappa_{\SS}$. We also consider the space $\H_\mathrm{m}(\kappa_{\SS})$ of marked surfaces of type $\kappa_{\SS}$ and we denote by $\pi_{\SS}: \H_{\mathrm{m}}(\kappa_\SS) \to \H_\SS$ the canonical projection that forgets the marking. We define the forgetful map
\begin{align}\label{forgetful map}
   p_{\SS}: \H \to \H_{\mathfrak{S}}, \ \ q \mapsto (X_q,\omega_q,{\mathfrak{z}_q}|_{\SS}),
\end{align}
and the corresponding map at the level of marked surfaces 
\begin{align}\label{lift of the forgetful map}
   \tilde{p}_{\SS}: \H_\mathrm{m} \to \H_\mathrm{m}(\kappa_{\mathfrak{S}}), \ \ (q,f) \mapsto ((X_q,\omega_q,{\mathfrak{z}_q}|_{\SS}), f).  
\end{align}

\noindent These two maps satisfy 
\begin{align}\label{eq : p and its lift}
    \pi_\SS \circ \tilde{p}_\SS = p_{\SS} \circ \pi,  
\end{align}
\noindent and we have: 
\begin{prop}\label{prop: connected fibers}
    The fibers of $\tilde{p}_\SS$ are connected. 
\end{prop}

\begin{proof}
    Let $(q_1,f_1)$ and $(q_2,f_2)$ be two marked surfaces in $\Hm$ that have the same image under $\tilde{p}_\SS$. By definition, this means that there is biholomorphism $\varphi : X_{q_1} \to X_{q_2}$ such that $\varphi^* \omega_{q_2} = \omega_{q_1}$ and $\varphi \circ {\mathfrak{z}_{q_1}}|_\SS = {\mathfrak{z}_{q_2}}|_\SS $, as well as an isotopy $(\psi_t : S \to S)_{t \in [0,1]}$ that connects the identify map of $S$ to $f_2 \circ \varphi \circ f_1$ without moving the points in $\SS$. Let $\mathfrak{z}_t : \Sigma \to X_{q_2}$ be the map $(f_2 \circ \psi_t)|_\Sigma$. Notice in particular that this map coincides with $\mathfrak{z}_2$ on $\SS$. Define 
    $$
    \gamma : [0,1] \to \Hm, \ t \mapsto (X_{q_2},\omega_{q_2}, \mathfrak{z}_t,f_2 \circ \psi_t).
    $$ 
    The path $\gamma$ connects $(q_2,f_2)$ to $(q_1,f_1)$  while staying in a single fiber of $\tilde{p}_\SS$. 
\end{proof}

If we let $\Res_{\SS} : H^1(S,\Sigma;\C) \to H^1(S,\SS;\C)$ be the canonical restriction map induced by the inclusion $\SS \subset \Sigma$, we have 
\begin{align}\label{eq: equivariance of dev}
    \dev \circ \tilde{p}_{\SS} = \Res_\SS \circ \dev,
\end{align}
\noindent where we denote the developing maps of $\H_{\mathrm{m}}$ and $\H_\mathrm{m}(\kappa_{\mathfrak{S}})$ by the same symbol $\dev$. 
Denote by $\mathfrak{R}_\SS$ the kernel of $\Res_\SS$ and let 
$$
Z_{\SS} \df \mathfrak{R}_{\SS} \cap Z.
$$ 
If $v \in H^1(S,\Sigma;\C)$, we will denote by $v|_{\SS}$ the image of $v$ under $\Res_\SS$. Let $\{\xi^*: \xi \in \Sigma\} \subset \mathfrak{R}$ be as in \S\ref{subsec: rel foliation}. Any element of $\mathfrak{R}_\SS$ (resp. $Z_\SS$) is a linear combination with complex (resp. real) coefficients of the $\xi^*$ where $\xi$ ranges over $\Sigma \sm \SS$. 

If $\xi \in \SS,$ we can think of $\xi^*$ simultaneously as an element of $H^1(S,\Sigma;\C)$ and $H^1(S,\SS;\C)$, and this identification defines 
a natural embedding 
\begin{equation}\label{eq: iota def}
\iota : \ker \left( H^1(S, \SS; \R^2)\to H^1(S; \R^2) \right)\to \mathfrak{R};
\end{equation}
that is, between the Rel subspaces corresponding to $\H_\SS$ and $\H$.  

\begin{prop}\label{relating the measures}
    Let $\N \subset \H_{\SS}$ be an invariant subvariety. Then $\M \df p_{\SS}^{-1}(\N)$ is an invariant subvariety. Furthermore, if $\mu$ is a $U$-invariant ergodic Radon probability measure on $\M$ that assigns zero mass to the set of surfaces with horizontal saddle connections, is invariant under $\Rel_{v}$ for any $v \in Z_{\SS}$, and such that ${p_{\SS}}_*\mu = m_{\N}$, then $\mu=m_{\M}$.  
\end{prop}

\begin{proof}
We start by proving that $\M$ is an invariant variety. Let $Y \subset \H_{\mathrm{m}}(\kappa_{\SS})$ be an irreducible component of $\pi_\SS^{-1}(\N)$ and let $W$ be the linear subspace of $H^1(S,\SS;\C)$ on which it is modeled. Let $V \subset H^1(S,\Sigma;\C)$ be the pre-image of $W$ under the map $\Res_\SS$ and let $X$ be the pre-image of $Y$ under $\tilde{p}_\SS$. We claim that $X$ is an equilinear manifold. Indeed, it follows from Proposition \ref{prop: connected fibers} that $X$ is connected and it follows from \eqref{eq: equivariance of dev} that $X$ is open in $\dev^{-1}(V)$. Since $X$ is closed, we conclude that $X$ coincides with a connected component of $\dev^{-1}(V)$. Since $\Mod(S,\Sigma)$ acts trivially on $\mathfrak{R}$, the induced action on $V$ of the stabilizer in $\Mod(S,\Sigma)$ of $X$ acts by determinant 1 endomorphisms. This proves that $X$ is an equilinear manifold. We deduce from \eqref{eq : p and its lift} that 
$$
\pi^{-1}(\M) = \bigcup \tilde{p}_\SS^{-1}(Y),
$$
\noindent where $Y$ ranges over the set of irreducible components of $\pi_\SS^{-1}(\N)$. Furthermore, this union is locally finite as the collection of irreducible components of $\pi_\SS^{-1}(\N)$ is itself locally finite. This proves that $\M$ is an invariant variety, and it is easily verified that any irreducible component of $\pi^{-1}(\M)$ is of the form $\tilde{p}_\SS^{-1}(Y)$, where $Y$ is an irreducible component of $\pi_\SS^{-1}(\N)$.  

The second assertion follows from Proposition \ref{prop: from SSWY}. To see this, 
note that the embedding $\iota$
defined in \eqref{eq: iota def}
induces an embedding 
of the spaces $\mathfrak{R}_{\N}$ and $Z_{\N}$  in the spaces $\mathfrak{R}_{\M}$ and $Z_{\M}$ respectively, satisfying 
\begin{equation}\label{eq: iota first}
p_{\SS} \circ \Rel_{\iota(x)} = \Rel_x \circ p_{\SS} \text{ for all } x \in Z_{\N},
\end{equation}
and
\begin{equation}\label{eq: iota second}
Z_\M = \iota( Z_\N) \oplus Z_\SS.
\end{equation}

We have assumed that $\mu$ is invariant under $Z_\SS$, and that $p_{\SS*}\mu = m_\N$, which is invariant under $Z_\N.$ Considering the disintegration of $\mu$ into its conditional measures with respect to the map $p_\SS$, which is unique up to a null set, and using \eqref{eq: iota first} and \eqref{eq: iota second}, 
 we see that $\mu$ is invariant under $\Rel_x$ for any $x \in Z_\M.$
% Let $X \df \tilde{p}_\SS^{-1}(\N)$ be an irreducible component of $\pi^{-1}(\M)$, let $V$ be the linear subspace of $H^1(S,\Sigma;C)$ on which $X$ is modeled and write
% $$
% V = V_x \oplus V_y
% $$
% \noindent where $V_x \df V \cap H^1(S,\Sigma;\R_x)$ and $V_y \df  V \cap H^1(S,\Sigma;\R_y)$. Following \cite{SSWY}, we define 
% $$
% \Psi : V^{(1)}  \sm \{0\} \to \mathbf{P}(V_x) \times V_y, \ (x,y) \mapsto ([x],y)
% $$
% We recall that in the terminology used in \cite{SSWY}, a {\em box in $X$} is a diffeomorphism 
% $$
% \varphi : U'_x \times U_y \to X^{(1)} \df X \cap \H_\mathrm{m}^{(1)}
% $$
% \noindent that satisfies $\Psi \circ \dev \circ \varphi = id$, where $U'_{\mathrm{x}}$ is an open subset in $\mathbf{P}(V_x)$ and $U_y$ is an open subset in $V_y$. Let $\beta$ A Radon measure $\mu_X$ supported on $X^{(1)}$ is said to be {\em horospherical} if for any box $\varphi : U'_x \times U_y \to X^{(1)} \df X \cap \H_\mathrm{m}^{(1)}$ and any compactly supported continuous function $f : X \to \R$, there is a measure $\lambda$ on $U_y$ and 
 \end{proof}

\begin{thm}\label{heredity of weak cassification}
Let $\N \subset \H_{\SS}$ be an invariant subvariety and let  $\mu$ be an ergodic $U$-invariant measure on $\H$ such that ${p_{\SS}}_* \mu = m_{\N}$. Then there is an invariant subvariety $\M \subset \H$ and $x \in Z$ such that $\mu =  {\Rel_x}_*m_\M$. 
\end{thm}

The proof will require the drift argument presented in \S \ref{subsec: drift}. We will use some of the notations and computations used in the proof of Proposition \ref{drift}. 
\begin{proof}
Let $\mu$ be an ergodic $U$-invariant measure, let $N_{\mu}$ be the group as in \eqref{eq: definition of the normalizer} and let $N_{\mu}^{\circ}$ be its connected component of identity. We will first prove Theorem \ref{heredity of weak cassification} under the  additional assumption that
\begin{align}\label{eq: deal first with}
   N^\circ_{\mu} \cap Z_\SS = \{0\}.
\end{align}

%Let $\| \cdot \|$ be a norm on $\mathfrak{R}$. %Let $\varepsilon_0 > 0$ be such that for any $x %\in Z_\SS$ with $\|x\| < \varepsilon_0$, we have %${\Rel_{x}}_{*}\mu \neq \mu$. Such a number %$\varepsilon_0$ exists by \eqref{eq: deal first %with} and the fact that $N^\circ_{\mu}$ is %closed.  
We claim that for any $\varepsilon> 0$, we can find a compact set $L \subset \H$ with $\mu(L) > 1-\varepsilon$ that intersects every fiber of $p_{\SS}$ in at most a finite set. Indeed, by inner regularity of $\mu$, let $K$ be a compact set in $\Omega_{\mu}$ of measure bigger than $1-\varepsilon$. Combining Egorov's theorem and Birkhoff's ergodic theorem, we can find a compact set $L \subset \Omega_{\mu}$ and  $T_0>0$ so that for any $T>T_0$ and any $q \in L,$
\begin{align}
    \frac{1}{T}|\{s \in [0,T] :  u_s q \in K \}| > 1-2\varepsilon.
\end{align}
Suppose there is $q \in \N$ such that $L_q \df L \cap p_{\SS}^{-1}(\{q\})$ is infinite. By compactness, there is sequence of points $q_n \in L_q$ converging to a point $q_{\infty} \in L_q$ with $q_n \neq q_\infty$. It follows from \eqref{eq: equivariance of dev} that for $n$ large enough, there is $w_n \in \mathfrak{R}_{\SS}$ such that $q_n = q_{\infty} \pluscirc w_n$ and $w_n \to 0$. It follows from Proposition \ref{prop: first} and  \eqref{eq: deal first with} that for large values of $n$ the vector $w_n$ is not in $Z_{\SS}$. If $\varepsilon$ is small enough, we can argue as in the proof of Proposition \ref{drift} and show that for any $n$ large enough, there is $s_n>0$ such that both $u_{s_n} q_n$ and $u_{s_n} q_{\infty}$ belong to $K$ and 
$$
\frac{\eta \varepsilon}{4} \leq \| x_n + s_n y_n \| \leq \eta \varepsilon,
$$ 
where $x_n$ and $y_n$ are such that $w_n = x_n + \mathbf{i} y_n$, $\| \cdot \|$ is some norm on $\mathfrak{R}_{\N}$, and $\eta$ is small enough so that $\Rel_{w}(q)$ is well-defined for any $w \in \mathfrak{R}$ with $\| w\| < \eta \varepsilon$. By \eqref{eq: equivariance Rel}, $$u_{s_n} q_n = u_{s_n} q_{\infty} \pluscirc (x_n + s_n y_n + \mathbf{i} y_n).$$
Using the compactness of $K$, by taking the limit of a convergent subsequence, as in the proof of Proposition \ref{drift}, we obtain  $q_{\infty}' \in \Omega_{\mu}$ and $x \in Z_{\SS}$ with $0 < \|x\| \leq \varepsilon$ such that $q_{\infty}' \pluscirc w \in \Omega_{\mu}$. By Proposition \ref{prop: first}, this implies that $\mu$ is invariant under $\Rel_{x}$. Since $\varepsilon$ can be taken arbitrarily small, this contradicts \eqref{eq: deal first with}, showing that $L_q$ is finite for $q \in \N.$

Now, let $\nu \df {p_\SS}_* \mu$ and let $q \mapsto \nu_q$ be the fiber-wise measures obtained by disintegration of $\mu$ along the fibers of $p_\SS$. The fact that $\mu$ is $U$-invariant implies that there is a $\nu$-conull set $E \subset \N$ such that for any $s \in \R$ and any $q \in E$ we have
\begin{equation}\label{eq: invariance decomposition}
{u_s}_{\ast}\nu_q = \nu_{u_sq}.  
\end{equation} 
Let $\nu_q = \nu_q^{(at)}+\nu_q^{(na)}$ be the decomposition of $\nu_q$ into its atomic and non-atomic parts. That is, $\nu_q^{(at)}$ is the restriction of $\nu_q$ to its collection of atoms. By \eqref{eq: invariance decomposition} and the ergodicity of $\nu$ we find that for $\nu$-a.e.\;\;$q$, either $\nu_q=\nu_q^{(at)}$ or $\nu_q= \nu_q^{(na)}.$ Suppose by contradiction that the latter holds, and consider once again the compact set $L$. We have shown that $\nu(L)>0$ and  $L_q$ is finite for any $q \in \N$.
Then we have 
%We define a measurable function $f %:q \mapsto \nu_q^{(0)}(\M)$, where %$\nu_q^{(0)}$ is the atomic part %of $\nu_q$.  It follows from %Equation \eqref{eq: invariance %decomposition} that $f$ is %$U$-invariant and thus by %ergodicity of $\nu$ with respect %to $U$, that $f$ is essentially %constant. We denote by $\eta$ this %constant and up to taking a %smaller $\nu$-conull $E$ we can %assume that $f|_{E} = \eta$. If %$\eta<1$, let $L$ be a compact %subset of $\H$ as above with %$\mu(L)>\eta$ and such that for %any $q \in \N$, the set $L \cap %p_\SS^{-1}\{q\}$ is finite. We %thus have 
$\nu_q(L) = \nu_q^{(na)}(L) =0$ for $\nu$-a.e.\;\;$q,$ and thus 
%\begin{align*}
 $  \mu(L) = \int_{\N} \nu_q(L) \ d\nu(q) = 0,
$
%\end{align*}
a contradiction. 

Thus there is a conull set $E \subset \N$
such that $\nu_q$ is an atomic measure whenever $q \in E$. By ergodicity and Equation \eqref{eq: invariance decomposition} it also follows that the number of atoms and their mass can be assumed to be constant on $E$. We denote by $N$ this number of atoms, so each $q \in E$ has mass $1/N$, and we denote by $\left\{q^{(1)},\ldots,q^{(N)} \right\}$ the atoms of $\nu_q$.   

We now claim that for any compact $L$ of $\M$ with $\mu(L)>1-\frac{1}{N}$, 
there are $q \in L$, a sequence of times $s_n >0$ converging to $0$ and a sequence of 
$w_n \in 
\mathfrak{R}_{\SS}$ 
converging to $0$ such that 
\begin{equation}\label{eq: in the form}
q_n = v_{s_n} q \pluscirc w_n \in L,
\end{equation}
where $v_s = \exp  \left(\begin{smallmatrix} 0 && 0 \\ s && 0 \end{smallmatrix} \right)$. Indeed, let $L$ be such a compact set. For any $q \in E$ we either have $\nu_q(L) = 1$ or $\nu_q(L) \leq 1-\frac{1}{N}$. Since $\mu(L)>1-\frac{1}{N}$, the set $\{q \in E  : \nu_q(L)=1 \} $ must have positive measure and by Luzin's theorem, it contains a set $E_0$ of positive measure so that the restriction of the measurable function 
$q \mapsto \left\{q^{(1)},\dots,q^{(N)}\right\}$ to $E_0$ is continuous. As the measure $\nu$ is $G$-invariant, the set 
$$E_1 \df \left\{q \in E_0 : v_{-\frac1n} q \in E_0 \ \text{for infinitely many} \ n  \right\}$$
has positive measure. Let $q \in E_1$, so that $q \in E_0$ and there is a decreasing sequence $s_n$ converging to $0$ such that $q_n \df v_{s_n} q \in E_0$. By definition of $E_1$, for any $j \in \{1,\dots,N\}$ the points $q_n^{(j)}$ are in $L$ and so is $q^{(1)}$. Since $p_\SS$ is $G$-equivariant, we know that $v_{s_n} q^{(1)}$ is in the same fiber of $p_\SS$ as $\left\{q_n^{(1)}, \dots, q_n^{(N)}\right \}$. By continuity and up to extracting a further subsequence, there is a $j_0 \in \{1,\dots,N\}$ and $w_n \in Z_\SS$ such that $q_n^{(j_0)} = v_{s_n} q^{(1)} \pluscirc w_n$ and $w_n \to 0$, which proves the claim.

We now employ the drift argument. Namely we take $K, K_0$ as in  the proof of Proposition \ref{drift}. We can assume that 
the compact set $L$ of the previous paragraph satisfies $L \subset K_0$ and that the sequence $q_n$ in $L$ is as in \eqref{eq: in the form}. Then the drift argument shows that 
the measure $\mu$ is invariant under a one-parameter group $\rho$ satisfying \eqref{eq: one parameter group rho}. 
%\begin{align*}
%     \rho(t)=\Bigl( %\left(\begin{smallmatrix} e^{at} %&& 0 \\ 0 && e^{-at} %\end{smallmatrix} \right), e^{at} %w - w \Bigr) \in B \ltimes %Z_{\SS}, \ \ \text{ where } (a, %w) \neq (0,0).
% \end{align*}  
Because $w \in Z_{\SS}$ and by \eqref{eq: deal first with} we have $a \neq 0$. Using $B$-invariance and \cite{eskin2018invariant} as in the proof of Theorem \ref{classification}, we see that there is $x \in Z$ such that $\mu={\Rel_x}_{\ast} m_{\M}$ for some invariant subvariety $\M \subset \H$. This concludes the proof under the assumption \eqref{eq: deal first with}.

When assumption \eqref{eq: deal first with} does not hold, we proceed as follows. Let $ \SS' \subset \Sigma$ be a subset of minimal cardinality containing $\SS$, and for which 
\begin{equation}\label{eq: for which}
N^\circ_{\mu} \cap Z_{\SS'} = \{0\} .
\end{equation}
Condition \eqref{eq: for which} is clearly satisfied for $\SS' = \Sigma,$ so the collection of subsets for which \eqref{eq: for which} holds is not empty. We define the map 
$$
p_{\SS',\SS} : \H_{\SS'} \to \H_{\SS}, \ q \mapsto (X_q,\omega_q,{\mathfrak{z}}|_{\SS}).
$$
Notice that this map is exactly the map defined in \eqref{forgetful map} in the case where the source is the stratum $\H(\kappa_{\SS'})$. By construction, we have $p_{\SS} = p_{\SS',\SS} \circ p_{\SS'}$. 
Let $\N' \df p^{0-1}_{\SS', \SS}(\N).$ 
It follows from Proposition \ref{relating the measures} applied to $\H_{\SS'}$ and $p_{\SS',\SS}$ (in place of $\H$ and $p_\SS$) that $\N'$ is an invariant subvariety. To conclude the proof of Theorem \ref{heredity of weak cassification}, it is enough to prove that ${p_{\SS'}}_* \mu=m_{\N'}$. Indeed, in that case we can repeat the argument presented in the first part of the proof, with $\SS'$ in place of $\SS$.

We now set $\nu \df {p_{\SS'}}_{\ast} \mu$ and prove  $\nu = m_{\N'}$. We shall use Proposition \ref{relating the measures} applied once again to $p_{\SS',\SS}$ in place of $p_{\SS}$. Since $p_{\SS',\SS} \circ p_{\SS'} = p_{\SS}$ and  by our  assumption that ${p_\SS}_* \mu = m_{\N}$, we have ${p_{\SS',\SS}}_* \nu = m_{\N}$.  
By the second assertion of the Proposition, it is enough to show that the measure $\nu$ is invariant under the flow $\Rel_{v}$ for any $v \in Z_{\SS',\SS}$, where 
$$
Z_{\SS',\SS} \df \ker \Big( H^1(S,\SS';\C) \to H^1(S,\SS;\C )\Big). 
$$
Let $\{\xi^* : \xi \in \Sigma\}$ be as in \S \ref{subsec: rel foliation}. Then the space $Z_{\SS',\SS}$ is spanned by the restriction to $H^1(S,\SS';\C)$ of the $\xi^*$ (which formally is the vector $\Res_{\SS'}(\xi^*)$) where $\xi$ ranges over the points in $\SS' \sm \SS$. Let $\xi_0 \in \SS' \sm \SS$. By the minimality of $\SS'$, we have that 
$$
\dim \big(N_{\mu}^{\circ} \cap Z_{\SS'\sm \{\xi_0\}}\big)>0.
$$
Let $x \in N_{\mu}^{\circ} \cap Z_{\SS'\sm \{\xi_0\}}$ be a non-zero element. We can write 
$$
x = \lambda_{\xi_0} \xi_0^* + \sum_{\xi \in \Sigma \sm \SS'} \lambda_{\xi}  \xi^*
$$
\noindent and we claim that $\lambda_{\xi_0} \neq 0$. Indeed, otherwise we would have that $x \in Z_{\SS'}$ which, in combination with \eqref{eq: for which}, would imply $x = 0$. By rescaling we can assume $\lambda_{\xi_0} = 1$. Now, notice that the restriction to $H^1(S,\SS';\C)$ of $\sum_{\xi \in \Sigma \sm \SS'} \lambda_{\xi}  \xi^*$ is equal to $0$ and thus we have $\Res_{\SS'}(x) = \Res_{\SS'}(\xi_0^*)$. It follows from \eqref{eq: equivariance of dev} and the fact that $x \in N_{\mu}$ that
$$
{\Rel_{t\cdot \Res_{\SS}(\xi_0^*)}}_* \nu = {p_{\SS'}}_* ({\Rel_{t \cdot x}}_*\mu) = {p_{\SS'}}_*\mu = \nu,
$$
hence $\nu$ is indeed invariant by $\Rel_v$ for any $v \in Z_{\SS',\SS}$. 
\end{proof}

\begin{remark}
The special case of  Theorem \ref{heredity of weak cassification} in which $\N$ is a closed $\mathrm{GL}_2^+(\R)$-orbit is the main result of \cite{eskin2006unipotent}. In the same paper, the  authors deduce that the orbit closure of a translation surface that is a translation branched cover of a Veech surface satisfies the weak classification of $U$-invariant measures. We recall that a translation cover of $q_0$ is a pair $(q,\rho)$ where $\rho : X_q \to X_{q_0}$ is a holomorphic map such that $\rho^*\omega_{q_0} = \omega_q$ and that $q_0$ is a Veech surface if its orbit is closed. Analogously, it should be possible to deduce from  Theorem \ref{heredity of weak cassification}
that the orbit-closure of a translation branched cover of a surface in a rank one invariant subvariety of Rel dimension one, satisfies weak classification of $U$-invariant measures. This would require pinning down  the relationship, as dynamical systems for the $U$-action, between the spaces $p_\SS^{-1}(\N)$ of Proposition \ref{relating the measures} (adding marked points to surfaces in an invariant subvariety $\N$), and the space of branched translation covers of a fixed topological type of surfaces in $\N$.  
See \cite[Section 3]{SmillieWeiss} for a related discussion of  the {\em moduli space of a branched covers}. 

For $\M = p_\SS^{-1}(\N)$ as in Proposition \ref{relating the measures}, the invariant subvarieties $\M' \subset \M$ such that $p_\SS(\M') = \N$ correspond to finite $\mathrm{GL}_2^+(\R)$-equivariant configurations of marked points over $\N$. Such configurations have been classified, see \cite{thealgebraichull, AW_marked_points}.
\end{remark}
As a  corollary we obtain that adding marked points to invariant subvarieties that satisfy themselves the weak classification of $U$-invariant measures, results in more such invariant subvarieties. In particular, we have:

\begin{cor}\label{cor: more more more known examples}
Let $\M \subset \H$ be an invariant subvariety such $p_{\Sigma_0}(\M)$ has rank one and Rel-dimension one, where $\Sigma_0$ is as in \eqref{sigma0}. Then $\M$ satisfies the weak classification of $U$-invariant measures.
\end{cor}

\section{Nondivergence of the horocycle flow and frequency of passage
  near singular sets} 
The main results of this section are Proposition \ref{frequencyofpassage} and Proposition \ref{frequencyofpassagesaddleconnection}, which uniformly bound the amount of time a horocycle trajectory can spend either near the support of one of the measures described in \S \ref{push}, or near the set of surfaces with a horizontal saddle connection of bounded length. 
 A special case of Proposition \ref{frequencyofpassage} was obtained in \cite[\S 10]{bainbridge2016horocycle}. We also establish a similar result for large circles in Proposition \ref{prop: push of circles} that is used for the proof of Theorem \ref{thm: circle averages}. 

 \subsection{Singular sets and long horocycle trajectories}
By a well-known compactness criterion, the sets
$$
\{q \in \H: \text{any saddle connection on } q \text{ has length } \geq \varepsilon \}
$$
\noindent are an exhaustion of $\H$ by compact sets. We will be interested in giving upper bounds on the amount of time a horocycle trajectory can spend outside large compact sets. For measurable $A \subset \R$, we denote the Lebesgue measure of $A$ by $|A|.$
Let $\M$ be an invariant subvariety and let $\M_r$ denote the surfaces in $\M$ with no horizontal saddle connections of length less than $r$. 
\begin{thm}[Quantitative nondivergence of the horocycle flow, \cite{minsky2002nondivergence}]\label{quantitative} 
There are constants $C>0$, $\rho_0$ and $\alpha>0$ such that for every $q \in \M$, every $\rho \in (0,\rho_0]$ and every $T>0$, if for every saddle connection $\delta$ on $q$ we have $\max_{s \in [0,T]} \|\hol_{u_s  q}(\delta) \| \geq \rho$, then for any $0<\varepsilon<\rho$, 
\begin{align}
    | \{s \in [0,T]  :  u_sq \ \text{has a saddle connection of length} < \varepsilon \} | < C\left(\frac{\varepsilon}{\rho}\right)^{\alpha}T .
\end{align}
In particular,
\begin{itemize}
    \item[(I)]For any $\varepsilon>0$ and any compact $L \subset \M$, there is a compact $K$ such that for any $T>0$ and any $q \in L$,        
        \begin{align*}
            \frac{1}{T} |\{s \in [0,T]  :  u_s  q \notin K \}| < \varepsilon.
        \end{align*}  
    \item[(II)] For any $\varepsilon>0$ and $r>0$, there is a compact $K \subset \M$ such that for any $q \in \M_r$ there is a $T_0>0$ such that for all $ T>T_0$,       
            \begin{align*}
                \frac{1}{T} |\{s \in [0,T]  :  u_s q \notin K \}| < \varepsilon.
            \end{align*}     
\end{itemize}
\end{thm}
It will be useful to introduce a scalar product $\langle \cdot | \cdot \rangle$ on $Z_{\M}$ and we denote by $\| \cdot \|$ the induced norm on $Z_{\M}$. For $R>0$ and an invariant subvariety $\N \varsubsetneq \M$, we denote
\begin{equation}\label{eq: def BR}
    \B(R) \df \left\{x \in Z_{\N}^{\perp} :  \|x\| \leq R\right \}.
 \end{equation}
From \cite[Thm. 11.2]{MW2} or \cite[Thm. 6.1]{bainbridge2016horocycle} we have: 

\begin{lem}\label{speed}
There is a constant $\sigma >0$, depending on the scalar product on $Z_\M$, such that for any $R>0$, if $x \in Z_{\M}$  with $\|x\|<R$ and $q \in \M_{\sigma R}$ then $\Rel_{x}(q)$ is well-defined.  
\end{lem}

\begin{prop}\label{injectivity}
Let $\N \subset \M$ be an invariant subvariety and let  $Z'$ be any linear subspace of $Z_{\M}$ transverse to $Z_{\N}$. Then for any $R>0$, there is $r>0$ such that the map
\begin{align}\label{eq: the map Phi2}
    \N_r \times \{x \in Z' : \|x\| \leq R \} \to \M,  \ \ \ \ \ (q,x) \mapsto \Rel_{x}(q)
\end{align}
is well-defined and injective. 
\end{prop}

\begin{proof}
Let $K$ be a compact subset of $\M$ as in item (II) of Theorem \ref{quantitative} for $r=1$ and $\varepsilon = \frac{1}{2}$. Since $K$ is compact, there is $\ell \in (0,1)$ such that any surface $q \in K$ does not have horizontal saddle connections shorter than $\ell$. For any $x \in Z'$, the vector field $\vec{x}$ is transverse to $\N$. From this, and from Lemma \ref{speed}, we find that there is $\varepsilon' < \frac{\ell}{\sigma}$ such that the map  
$$
K \times \{x \in Z' : \|x\| \leq \varepsilon' \} \to \M, \ \ \ \ \ (q,x) \mapsto \Rel_{x}(q)
$$
is well-defined and injective. 

 Given $R>0$, let $r \df  \frac{R}{\varepsilon'}$ and $t \df \log r$, so that $e^t\varepsilon' = R \text{ and } r = e^t.$   Suppose $q, q' \in \N_{r}$ and  $x, x' \in \{z \in Z' : \|z\| \leq R \}$ are such that $\Rel_{x}(q)=\Rel_{x'}(q')$ (it follows from Lemma \ref{speed} and the choice of $\varepsilon'$  that $\Rel_{x}(q)$ and $\Rel_{x'}(q')$ are well-defined). By \eqref{eq: saddle connection equivariance} and \eqref{eq: some matrices}, one has $g_{-t} q, g_{-t} q' \in \N_1$, and thus, by definition of $K$ there is $s>0$ such that $u_sg_{-t} q$ and $u_sg_{-t}q'$ both belong to $K$. By \eqref{eq: equivariance Rel}, 
\begin{align*}
    \Rel_{e^{-t}x}(u_sg_{-t} q)=u_s g_{-t} \Rel_x(q)= u_s g_{-t} \Rel_{x'}(q') = \Rel_{e^{-t}x'}(u_sg_{-t} q').
\end{align*}
\noindent By definition $e^{-t}x$ and $e^{-t}x'$ belong to $\{z \in Z' : \|z\| \leq \varepsilon' \}$, and thus $x=x'$ and $q = q'$. 
\end{proof}

\begin{remark}
Let $\N \subset \M$ be an invariant subvariety. Proposition \ref{injectivity} implies that for any $R>0$, there is $r>0$ such that the map
\begin{align}\label{eq: the map Phi}
    \N_r \times \B(R) \to \M,  \ \ \ \ \ (q,x) \mapsto \Rel_{x}(q)
\end{align}
\noindent is well-defined and injective.
\end{remark}

The following useful statement was derived from Theorem \ref{quantitative} in \cite[Prop. 10.5]{bainbridge2016horocycle}.

\begin{prop}\label{supernondivergence}
Let $L$ be a compact subset of $\M_{\infty}$, $\varepsilon>0$ and $r>0$. There is an open neighborhood $\mathcal{W}$ of $L$ and a compact $\widehat{\mathcal{W}} \subset \M_r$ containing $\mathcal{W}$ such that for any $q \in \H$ and any interval $I \subset \R$ for which $u_s q \in \mathcal{W}$ for some $s \in I$, we have 
\begin{align*}
  \left | \left\{ s \in I  : u_s q \notin \widehat{\mathcal{W}} \right\} \right| < \varepsilon |I|.
\end{align*}
\end{prop}

By combining Proposition \ref{injectivity} and Proposition \ref{supernondivergence}, we get the following:  

\begin{prop}\label{supernondivergenceinjectivity} 
Let $\N \subset \M$ be an invariant subvariety. For any $R>0$, $\varepsilon>0$ and any compact $K \subset \N_{\infty}$, there is  $\delta>0$, a neighborhood $\mathcal{W}$ of $K$ and a compact set $\widehat{\mathcal{W}} \subset \N$ containing $\mathcal{W}$, such that:

\begin{itemize}
    \item the map 
        \begin{align*}\label{eq: is well defined} 
            \widehat{\mathcal{W}} \times \B(R) \times \B(\delta) \to \M, \ \ \ \ \   (q,x,y) \mapsto \Rel_{x+\mathbf{i} \cdot y}(q)
        \end{align*}
        is well-defined, continuous and injective;
    \item for any interval $I \subset \R$ and any $q \in \M$ the following is satisfied: if $s \in I$ is such that $u_s  q \in \mathcal{W}$, then
        \begin{align*}
            \left| \left\{ s \in I  : u_s q \notin \widehat{\mathcal{W}} \right\} \right| < \varepsilon |I|.
        \end{align*}
\end{itemize}
\end{prop}

\begin{proof}
By Proposition \ref{injectivity}, there exists $r>0$ such that the map 
$\Phi :\N_r \times \B(R) \to \M$ defined by \eqref{eq: the map Phi} is well-defined and injective. 
Let $\mathcal{W}$ and $\widehat{\mathcal{W}}$ be the sets provided by Proposition \ref{supernondivergence}, corresponding to this $r$ and to $L=K$, where $K$ is as in the statement of the Proposition. Since $\widehat{\mathcal{W}} \subset \N_r$, it follows that the restriction of $\Phi$ to $\widehat{\mathcal{W}} \times \B(R)$ is still injective. Since $\widehat{\mathcal{W}}$ is compact, and since for any $y \in  Z_{\N}^{\perp}$, $\mathbf{i}\cdot y$ is transverse to $\N$, there is $\delta_0>0$ such that for any $\delta \leq \delta_0$, the map 
\begin{align*}%\label{eq: is well defined1} 
     \widehat{\mathcal{W}} \times \B(R) \times \B(\delta)  \to \M, \ \ \ \ \   (q,x,y) \mapsto \Rel_{\mathbf{i} \cdot y}(\Rel_{x}(q))
\end{align*}
\noindent is well-defined, continuous and injective. Finally, it follows from \cite[Prop. 4.5]{bainbridge2016horocycle} that for any $(q,x,y) \in   \widehat{\mathcal{W}} \times \B(R) \times \B(\delta)$ we have $\Rel_{\mathbf{i} \cdot y}(\Rel_{x}(q)) = \Rel_{x+\mathbf{i} \cdot y}(q)$, which concludes the proof.
\end{proof}
For any $r>0$ and any $K \subset \N$, denote
$$
Z(K,R) \df \left\{ \Rel_x(q) : q \in K, \, x \in \B(R) \cap Z^{(q)} \right\}.
$$

\begin{prop}[Frequency of passage near Rel pushes of invariant subvarieties]\label{frequencyofpassage}
Let $\N \varsubsetneq \M$ be an invariant subvariety, let $r>0$, let $K$ be a compact subset of $\N_\infty$ and let $\varepsilon>0$. There is a neighborhood $\mathcal{U}$ of $Z(K,r)$ and $R>r$ such that for any compact subset $L$ of $\M_{R} \sm Z(\N,R)$, there is $T_0>0$ such that for any $T>T_0$ and any $q \in L$, we have
\begin{align}
    \frac{1}{T} \ |\{s\in [0,T]  : u_s  q \in \mathcal{U} \}| <\varepsilon    . 
\end{align}

\end{prop}

\begin{proof}
Let $R'>r$ be large enough so that
\begin{align}\label{eq: def of R'}
    R' > \frac{(8 + \varepsilon)(r+1)}{\varepsilon}.
\end{align}
    By Proposition \ref{supernondivergenceinjectivity}, there is an open neighborhood $\mathcal{W}$ of $K$, a compact $\widehat{\mathcal{W}} \subset \N$ that contains $\mathcal{W}$, and $\delta>0$, such that the map 
\begin{align*}
    \bar \Phi: \widehat{\mathcal{W}} \times \B(R') \times \B(\delta) \to \M, \ \ \bar \Phi (q,x,y) \df \Rel_{x+\mathbf{i} \cdot y}(q)
\end{align*}
\noindent is well-defined, continuous and injective, and such that for any interval $I \subset \R$ and $s\in I$ and $q \in \N$ with $u_s  q \in \mathcal{W}$, we have
\begin{equation}
    \label{eq: the interior}
    \left| \left\{ s\in I  :  u_s q \notin \widehat{\mathcal{W}} \right\} \right| < \frac{\varepsilon}{4}|I|.
\end{equation}
Denote the interior of $\B(R')$ by  $ \mathrm{int}(\B(R')) $ (that is, use strict inequality in \eqref{eq: def BR}), and let
\begin{align*}
    \mathcal{V} \df \mathcal{W} \times \mathrm{int} (\B(r+1)) \times \mathrm{int}(\B(\delta)) \ \ \ \text{ and } \ \ \ \mathcal{U} \df \bar \Phi(\mathcal{V}). 
\end{align*}
Since $\M$ has rank one, and since $\bar \Phi$ is a homeomorphism onto its image, $\mathcal{U}$ is open in $\M$. It contains $Z(K,r)$ by construction. Let $q \in \M \sm Z(\N,R')$ and let
$$
    I_{\mathcal{U},q} \df \{t > 0 : u_t q \in \mathcal{U} \}.
$$

By injectivity of the map $\bar \Phi$, for any $t \in I_{\mathcal{U},q}$ there is a unique triple $(q_t,x_t,y_t) \in \mathcal{V}$ such that $u_t  q = \bar \Phi(q_t,x_t,y_t)$. Since $q \notin Z(\N, R')$ we have $y_t \neq 0$. For any $t \in I_{\mathcal{U},q}$, let 
\begin{align}\label{eq: def of IJ}
    \mathcal{J}_t &= t + \{s \in \R  :  \|x_t + sy_t\| \leq R'\} \\
    \mathcal{I}_t &= t + \{s \in \R  :  \|x_t + sy_t\|< r+1 \}.
\end{align}
Since 
$\mathcal{I}_t$ and $\mathcal{J}_t$ are sub-level sets of real quadratic polynomials, they are intervals.

We claim that if $t \in I_{\mathcal{U},q}$ and $I$ is an interval that contains one of the endpoints of $\mathcal{J}_t$, we have
\begin{equation}\label{eq: our claim}
    |\mathcal{I}_t \cap I|< \frac{\varepsilon}{4} |\mathcal{J}_t \cap I|.
\end{equation}
Indeed, define $p(s) \df \|x_t + sy_t\|^2 = \| x_t \|^2  + 2s \langle x_t,y_t \rangle + s^2 \| y_t \|^2 $ and for any $C\geq 0$, define 
\begin{align}\label{eq: def of Delta}
\Delta_C \df 4\langle x_t,y_t \rangle^2 - 4\| y_t \|^2  (\| x_t \|^2 - C^2)).
\end{align}
This is the discriminant of the real quadratic polynomial whose zeroes are points at which $\|x_t +sy_t\|=C.$ Since $\|x_t + sy_t\| \geq 0$ for all $s$, we have $\Delta_0 \leq 0$ and since $R'>r+1$, by the assumption that $I$ contains one of the endpoints of $\mathcal{J}_t$ we have $\Delta_{r+1} > 0$. From the formula for the roots of a quadratic polynomial, for any $C \geq r+1$ we have
\begin{equation*}
    p(s) \leq C^2 \ \Longleftrightarrow \ s \in - \frac{\langle x_t,y_t \rangle }{\| y_t\|^2} + \bigg[\frac{- \sqrt{\Delta_C}}{2\| y_t\|^2}, \frac{ \sqrt{\Delta_C}}{2\| y_t\|^2}\bigg].
\end{equation*}  
By construction, we have that $s+t \in \mathcal{J}_t$ if and only if $p(s) \leq R'^2$ and $s+t \in \mathcal{I}_t$ if and only if $p(s) < (r+1)^2$. Remembering that $\Delta_{r+1} = \Delta_0 + 4\|y_t \|^2(r+1)^2$ and that $\Delta_0 \leq 0$, we have 
\begin{equation}\label{eq: we compute}
\frac{|\mathcal{I}_t \cap I|}{|\mathcal{J}_t \cap I|} \leq \frac{ \sqrt{\Delta_{r+1}} }{ \| y_t\|^2 } \, \cdot \, \frac{2 \|y_t\|^2}{\sqrt{\Delta_{R'}} - \sqrt{\Delta_{r+1}} }.
\end{equation}
Indeed, \eqref{eq: we compute} is obvious in case $\mathcal{I}_t \cap I = \varnothing$; in case it is not, $I$ contains a subinterval bounded by an endpoint of $\mathcal{I}_t$ and an endpoint of $\mathcal{J}_t$, so 
an interval in which $P(s) \in \left((r+1)^2, R'^2 \right]$. This gives a lower bound for the denominator $|\mathcal{J}_t \cap I|$, and as an upper bound for the numerator we use $|\mathcal{I}_t|.$
We obtain 
\begin{align*}
    \frac{|\mathcal{I}_t \cap I|}{|\mathcal{J}_t \cap I|} &\stackrel{\eqref{eq: we compute}}{\leq} 
    %\frac{ \sqrt{\Delta_{r+1}} }{ %\| y_t\|^2 } \, \cdot \, %\frac{2 \|y_t\|^2}%{\sqrt{\Delta_{R'}} - %\sqrt{\Delta_{r+1}} }= 
    2 \frac{ \sqrt{\Delta_{r+1}} }{\sqrt{\Delta_{R'}} - \sqrt{\Delta_{r+1}} }\\
    &=  2\frac{\sqrt{\Delta_{r+1}}({\sqrt{\Delta_{R'}} + \sqrt{\Delta_{r+1}}})}{4\|y_t\|^2 (R'^2 - (r+1)^2)} 
    =  \frac{\Delta_{r+1} + \sqrt{\Delta_{r+1} \Delta_{R'}}}{2\|y_t\|^2 (R'^2 - (r+1)^2)} \\
   % &= \frac{\Delta_{r+1} + \sqrt{\Delta_0^2 + 4 \| y_t \|^2 \Delta_0(R'^2+(r+1)^2) + 16\|y_t\|^4R'^2(r+1)^2 }}{2\|y_t\|^2 (R'^2 - (r+1)^2)} \\
    &\leq \frac{\Delta_{r+1} + \sqrt{\Delta_0^2 + 16\|y_t\|^4R'^2(r+1)^2 }}{2\|y_t\|^2 (R'^2 - (r+1)^2)} \\
    &\leq \frac{\Delta_{r+1} - \Delta_0 + 4 \| y_t \|^2R'(r+1)}{2\|y_t\|^2 (R'^2 - (r+1)^2)} \\
    &\leq 2 \frac{(r+1)^2+R'(r+1)}{R'^2-r^2} =  \frac{2(r+1)}{R'-r} \stackrel{\eqref{eq: def of R'}}{<} \frac{\varepsilon}{4}.
\end{align*}
This proves \eqref{eq: our claim}. 

Now, we claim that if $s \in \mathcal{J}_t \sm \mathcal{I}_t$ is such that $u_{s}  q \in \mathcal{U}$ then $u_{s-t} q_t \notin \widehat{\mathcal{W}}$. Indeed, we have $u_s  q = \Rel_{x_t+(s-t)y_t+\mathbf{i} \cdot y_t}(u_{s-t} q)$ by \eqref{eq: equivariance Rel} and thus,  if $u_{s-t}q \in \widehat{\mathcal{W}}$, then $u_sq=\bar\Phi(u_{s-t}q,x_t+(s-t)y_t,y_t)$. But the fact that $s \notin \mathcal{I}_t$ yields a contradiction to the injectivity of $\bar\Phi$.

For $I$ containing an endpoint of $\mathcal{J}_t$, we deduce
\begin{align*}
| \mathcal{J}_t \cap I_{\mathcal{U},q} \cap I| = \ \ & |\mathcal{I}_t \cap I_{\mathcal{U},q} \cap I | + | \big( \mathcal{J}_t\sm\mathcal{I}_t \big) \cap I_{\mathcal{U},q} \cap I | \\ 
\leq \ \ &  |\mathcal{I}_t \cap I | + \left| \left\{s \in \mathcal{J}_t \cap I :  u_s q_t \notin \widehat{\mathcal{W}} \right \} \right| \\ 
\stackrel{\eqref{eq: our claim}, \eqref{eq: the interior}}{<} &\frac{\varepsilon}{4}|\mathcal{J}_t \cap I | + \frac{\varepsilon}{4}|\mathcal{J}_t \cap I |  = \frac{\varepsilon}{2}|\mathcal{J}_t \cap I|.
\end{align*}

Now, let $\sigma$ be as in Lemma \ref{speed}, let $c \df \max(1,\sigma)$ and let $R \df cR'$. We will show that $\mathcal{U}$ and $R$ satisfy the required conclusions. Indeed, let $L$ be a compact subset of $\M_R \sm Z(\N,R)$. For any $q \in L$, we will denote by $\mathcal{J}_{t} = \mathcal{J}_{q,t}$ the interval as in \eqref{eq: def of IJ} to stress the dependence on $q$. We claim that there is $\delta_0$ such that for any $q \in L$ and for any $t \in I_{\mathcal{U},q}$, if $0 \in \mathcal{J}_{q,t}$ then $\|y_t\| \geq \delta_0$. Indeed, suppose by contradiction that there is a sequence $q_n \in L$ together with $t_n \in I_{\mathcal{U},q_n}$ such that $0 \in \mathcal{J}_{q,t_n}$. Let $(q_n',x_n,y_n) \in \mathcal{V}$ such that $u_{t_n}q_n = \bar \Phi(q'_n,x_n,y_n)$ and assume that $\| y_n \| \to 0$. Since $0 \in \mathcal{J}_{q,t_n}$, we have 
$$ 
q_n = u_{-t_n}\bar \Phi(q'_n,x_n,y_n) = u_{-t_n} q'_n \pluscirc (x_n - t_n y_n + \mathbf{i} y_n)
$$
\noindent with $\| x_n - t_ny_n \| \leq R'$. Passing to a subsequence, we can assume that $q_n$ converges to some $q_{\infty} \in L$ and that $x_n -t_ny_n$ converges to some $x_{\infty}$ satisfying $\| x_{\infty }\| \leq R'$. In particular, $q_{\infty} \pluscirc (-x_{\infty},0)$ is well-defined by Lemma \ref{speed} and it follows from \cite[Prop 4.6]{bainbridge2016horocycle} that $q_n \pluscirc -(x_n - t_n y_n,y_n)$ is a sequence in $\N$ that converges to $q_{\infty} \pluscirc (-x_{\infty},0)$. Since $\N$ is closed, this shows that $q_{\infty}$ belongs to $Z(\N,R)$, which is a contradiction and our claim is proved. 

Let $T>0$, let $q \in L$ and let $t \in I_{\mathcal{U},q}$ satisfy $[0,T] \cap \mathcal{J}_{q,t} \neq \emptyset$. If $[0,T]$ does not contain one of the endpoints of $\mathcal{J}_{q,t}$, then it is contained in $\mathcal{J}_{q,t}$. In particular, $0 \in \mathcal{J}_{q,t}$ and thus $\|y_t\| \geq \delta_0$. It follows also that
$$
T < \inf(\mathcal{J}_{q,t}) + |\mathcal{J}_{q,t}| < |\mathcal{J}_{q,t}| <   \frac{\sqrt{\Delta_{R'}}}{\|y_t\|^2} \leq 2\|y_t\|R' \leq \frac{R}{\delta_0}.
$$ 

Let $T_0 \df \frac{R}{\delta_0}$. We proved that for any $q \in L$, any $T>T_0$ and any $t \in I_{\mathcal{U},q}$ such that $[0,T] \cap \mathcal{J}_{q,t} \neq \emptyset$, we have that $[0,T]$ contains one of the endpoints of $\mathcal{J}_{q,t}$. The set $[0,T] \cap I_{\mathcal{U},q}$ is covered by the sets $(\mathcal{J}_{q,t})_{t \in I_{\mathcal{U},q}}$. From this cover, we can extract a subcover $(\mathcal{J}_{j})$ such that any $q \in [0,T]$ is covered at most twice, and hence
\begin{align*}
    |[0,T] \cap I_{\mathcal{U},q}| &\leq \sum_{j} |[0,T] \cap I_{\mathcal{U}} \cap \mathcal{J}_{j} | \\ 
    &< \frac{\varepsilon}{2} \sum_{j} |[0,T] \cap \mathcal{J}_{j}| \leq \varepsilon T. 
\end{align*}
\end{proof}

\begin{prop}[Frequency of passage near the horizontal saddle connection locus]\label{frequencyofpassagesaddleconnection}
Let $\varepsilon>0$ and let $r>0$. There is an open set $\mathcal{U}$ that contains $\M \sm \M_r$ and $R>r$ such that for any compact subset $L$ of $\M_R$, there is $T_0>0$ such for any $T>T_0$ and any $q \in L$, we have
\begin{align}
\frac{1}{T} \ |\{s\in [0,T]  : u_s  q \in \mathcal{U} \}| < \varepsilon.   
\end{align}
\end{prop}

\begin{proof}
Let $C$, $\rho_0$ and $\alpha$ as in Theorem \ref{quantitative}, let $\varepsilon'<\rho_0$ be small enough so that $C(\frac{\varepsilon'}{\rho_0})^{\alpha}<\varepsilon$, let $K \subset \M$ be the compact set of surfaces whose shortest saddle connection has length $\geq \varepsilon'$, let $t>\log(\frac{r}{\varepsilon'})$, let $R = \rho_0e^t$ and finally let $\mathcal{U} = g_t(K^c)$. Notice that $\M \sm \M_r \subset \mathcal{U}$.  

Now let $L$ be any compact subset of $\M_R$. By compactness, there is $\eta>0$ such that for any $q \in L$ and any saddle connection $\delta$ on $q$ with $\hol_q(\delta)=(x,y)$, we either have $|x|> R$ or $|y|>\eta$. Let $T'_0$ be large enough so that $T'_0e^{t}\eta \geq 2\rho_0$. This condition ensures for any $q \in L$, the length of any saddle connection on $g_{-t}q$ will eventually become larger than $\rho_0$ after applying the horocycle flow up to time $T_0$. It follows from Theorem \ref{quantitative} that for any $T>T'_0$ and any $q \in L$
\begin{align*}
\frac{1}{T} \ |\{s\in [0,T]  : u_s  g_{-t} q \in K^c \}| < \varepsilon.  
\end{align*}
Finally, let $T_0 = e^{2t}T_0'$, let $q \in L$ and let $T>T_0$. Then we have
\begin{align*}
    &  \frac{1}{T} \int_0^T  \mathbf{1}_{\mathcal{U}}(u_s q) \, ds = \frac{1}{T} \int_0^T \mathbf{1}_{K^c}(g_{-t}u_s q)\, ds \\ 
    = &\frac{1}{Te^{-2t}} \left|\left\{s \in \left[0,Te^{-2t} \right] : u_s  g_{-t}q \notin K \right\} \right| < \varepsilon.
\end{align*}
\end{proof}

\subsection{Singular sets and large circles}

We extend the Euclidean norm on $Z_{M}$ to a Euclidean norm on $\mathfrak{R}_{\M}$, \textit{i.e.,} for any $v = x + \mathbf{i} \cdot y \in \mathfrak{R}_{\M}$, we let $\| v \| = \sqrt{ \|x\|^2 + \|y\|^2}$. Notice that $r_{\theta}$ acts on $\mathfrak{R}_{\M}$ by isometry.

\begin{prop}\label{prop: bound}
Let $\N \subset \M$ be an invariant subvariety and let $q \notin \N$. There exists $\tilde{\delta}>0$ such that for all $t > 0$, if there exists $\theta \in [0,2\pi]$ such that $g_tr_{\theta} q = q' \pluscirc v$ with $q' \in \N$ and $v \in \mathfrak{R}^{\perp}_{\N}$, then $ \|  g_{-t} v \| \geq \tilde{\delta}$. 
\end{prop}

\begin{proof}
Since $q \notin \N$ and $\N$ is closed, it follows by continuity of Rel flows that there is $\tilde{\delta} > 0$ such that for any $w \in  \mathfrak{R}^{\perp}_{\N}$ with $\|w\| < \tilde{\delta}$ we have that $\Rel_w(q)$ is well-defined and does not belong to $\N$. Let $t>0$ and assume that there are $\theta \in [0,2\pi]$ and $q' \in \N$ such that $g_tr_{\theta}q = q' \pluscirc v$. 
Let $w \df - r_{-\theta} g_{-t} v.$
If $\| g_{-t} v\| < \tilde{\delta}$ then we also have $\| w \| < \tilde{\delta}$, and thus $q \pluscirc w$ is well-defined. But by 
\eqref{eq: equivariance Rel}, 
$q \pluscirc w = r_{-\theta} g_{-t} q' \in \N$ and we get a contradiction with the definition of $\tilde{\delta}$. This implies that $\| g_{-t} v\| \geq \tilde{\delta}$. 
\end{proof}

Let $n$ be a positive integer. We consider the diagonal action of $\mathrm{GL}_2^+(\R)$ on $\C^n$ where $\mathrm{GL}_2^+(\R)$ acts on each copy of $\C \simeq \R^2$ by linear transformations. For any $v \in \C^n$, we let
\begin{align}\label{eq: def of ellipses}
    E_{t,v} = \{g_tr_{\theta} v : \theta \in [0,2\pi]\}.
\end{align} 
We will need the following:
\begin{lem}\label{lem: little computation}
For any $0\leq \eta_1 \leq \eta_2 \leq 2$ and $0<\varepsilon<1$ we have
\begin{equation}
    \arccos(\varepsilon^2\eta_1-1) - \arccos(\varepsilon^2\eta_2 - 1) \leq \varepsilon (\mathrm{arccos}(\eta_1 - 1) - \mathrm{arccos}(\eta_2 - 1)). 
\end{equation}
\end{lem}

\begin{proof}
Use $2 \varepsilon x -\varepsilon^2 x^2 > \varepsilon(2x - x^2)$ in conjunction with the following computation: 
\begin{align*}
    &  \arccos(\varepsilon^2\eta_1 - 1) - \arccos(\varepsilon^2\eta_2 - 1) = \int_{ \varepsilon^2 \eta_1-1}^{\varepsilon^2 \eta_2 -1} \frac{1}{\sqrt{1-x^2}} \ dx \\
     =& \int_{\varepsilon^2 \eta_1}^{\varepsilon^2 \eta_2} \frac{1}{\sqrt{2x -x^2}} \ dx 
     = \varepsilon^2 \int_{\eta_1}^{\eta_2} \frac{1}{\sqrt{2 \varepsilon^2 x -\varepsilon^4 x^2}} \ dx \\
     \leq & \varepsilon \int_{\eta_1}^{\eta_2} \frac{1}{\sqrt{2 x - x^2}} \ dx = \varepsilon(\arccos(\eta_1-1) - \arccos(\eta_2 - 1)).
\end{align*}
\end{proof}

For any $r>0$, let $B(r)$ be the ball of radius $r$ in $\R^n$. Let $T, \delta>0$ and define
\begin{align*} 
B(T,\delta) \df \{x+ \mathbf{i}\cdot y \in \C^n  : (x,y) \in B(T) \times B(\delta) \}.
\end{align*} 

\begin{prop}\label{prop: ellipses and boxes}
Let $v \in \C^n$, let $T,\delta>0$ and let $0<\varepsilon < 1$. If $E_{t,v}$ is not contained in $B(T,\delta)$ then 
\begin{align*}
    | \{ \theta \in [0, 2\pi]  :  g_t r_{\theta} v \in B(\varepsilon T, \varepsilon \delta)\} | < \varepsilon | \{ \theta \in [0, 2\pi]  :  g_t r_{\theta}  v \in B(T,\delta) \} |.
\end{align*}
\end{prop}

\begin{proof}
We first observe that, up to replacing $T$ and $\delta$ by $e^{-t}T$ and $e^t\delta$, we can assume that $t=0$ and up to replacing $v$ by $r_{\frac{\pi}{2}}v$ we can assume that $T \leq \delta$. Write $v = (x_1 + \mathbf{i} y_1,\dots,x_n + \mathbf{i} y_n)$ and $r_{\theta}v = (x_{\theta},y_{\theta})$. Let $R = \|v\|$ and let $\theta_0 \in [0,2\pi]$ be such that
\begin{align*}
    \cos(2\theta_0)\left( \sum_{i=1}^n x_iy_i \right) - \sin(2\theta_0)\left(\sum_{i=1}^n \frac{x_i^2 - y_i^2}{2} \right) = 0 \\
    \cos(2\theta_0)\left(\sum_{i=1}^n \frac{x_i^2 - y_i^2}{2} \right) + \sin(2\theta_0) \left( \sum_{i=1}^n x_iy_i \right) \geq 0.
\end{align*}
Let $r\geq 0$ such that 
$$r^2 = 2\left(\cos(2\theta_0)\left (\sum_{i=1}^n \frac{x_i^2 - y_i^2}{2} \right) + \sin(2\theta_0) \left( \sum_{i=1}^n x_iy_i \right)\right).$$ Notice that $r \leq R$ and
\begin{align*}
     \| x_{\theta + \theta_0} \|^2  &= \frac{r^2}{2}\cos(2 \theta) + \frac{R^2}{2} \\
     \| y_{\theta + \theta_0} \|^2 &= - \frac{r^2}{2}\cos(2 \theta) + \frac{R^2}{2} . 
     \end{align*}
The degenerate case where $r=0$ is trivial so we can assume $r>0$ and up to replacing $v$ by $r_{\theta_0}v$ we can assume that $\theta_0 = 0$. It follows that for any $T',\delta'>0$, we have that $r_{\theta} v \in B(T',\delta')$ if and only if 
\begin{align*}
   -\frac{2\delta'^2- R^2}{r^2} \leq \cos(2\theta)  \leq \frac{2T'^2- R^2}{r^2}.
\end{align*}
Let $\eta_1 \df -\frac{2\delta^2 - (R^2+r^2)}{r^2}$ and $\eta_2 \df \frac{2T^2 - (R^2-r^2)}{r^2}$. The assumption that $E_{0,v}$ is not contained in $B(T,\delta)$ implies that $\eta_2 \leq 2$. Notice that $\left[-\frac{2(\varepsilon \delta)^2- R^2}{r^2}, \frac{2(\varepsilon T)^2- R^2}{r^2} \right] \subset [\varepsilon^2\eta_1 - 1,\varepsilon^2\eta_2 - 1 ]$. Using Lemma \ref{lem: little computation},
\begin{align*}
    | \{ \theta \in [0, 2\pi] :  g_t r_{\theta} v \in B(\varepsilon T, \varepsilon \delta)\} | &< \arccos(\max(0,\varepsilon^2\eta_1 )-1) - \arccos(\varepsilon^2\eta_2 - 1) \\
    &< \varepsilon (\arccos(\max(0,\eta_1)-1) - \arccos(\eta_2 - 1)) \\
    &< \varepsilon | \{ \theta \in [0, 2\pi] \ | \ g_t r_{\theta} v \in B(T,\delta)\} |.
\end{align*}
\end{proof}

For any $r,\delta>0$ and any invariant subvariety $\N \subset \M$, we denote 
\begin{align*} 
B_{\N}^{\perp}(r,\delta) \df \{x+ \mathbf{i}\cdot y \in \mathfrak{R}_{\M}  :  (x,y) \in B_{\N}^{\perp}(r) \times B_{\N}^{\perp}(\delta) \}
\end{align*} 

\begin{prop}[Proportion of large circles near singular sets]\label{prop: push of circles}
Let $\N \varsubsetneq \M$ be an invariant subvariety, 
let $r>0$, let $K$ be a compact subset of $\N_{\infty}$ and let $\varepsilon>0$. 
There is a neighborhood $\mathcal{U}$ of $Z(K,r)$ such that for any $q \in \M \smallsetminus \N$, there is $t_0>0$ such that for any $t>t_0,$
\begin{align*}
    \frac{1}{2\pi} | \{ \theta \in [0,2\pi] : g_tr_{\theta}q \in \mathcal{U} \} | < \varepsilon.
\end{align*}
\end{prop}

\begin{proof}
Let $R>0$ be large enough so that
\begin{align*}
    R > \frac{8}{\varepsilon}r,
\end{align*}
and 
let $\tilde{\delta}$ be as in Proposition \ref{prop: bound}.  A straightforward extension of \cite[Prop 12.5]{bainbridge2016horocycle} to arbitrary rank one invariant subvarieties gives the following analogue of Proposition \ref{supernondivergenceinjectivity}:
There exist $t_0>0$, $\delta<\frac{\tilde{\delta}}{\sqrt{2}}$, a neighborhood $\mathcal{W}$ of $K$ and a compact set $\widehat{\mathcal{W}} \subset \N$ such that 
\begin{itemize}
\item
the map
\begin{align*}
    \bar \Phi: \widehat{\mathcal{W}} \times \B(R) \times \B(\delta) \to \M, \ \ \bar \Phi (q,x,y) \df \Rel_{x+\mathbf{i} \cdot y}(q)
\end{align*}
\noindent is well-defined, continuous and injective; 
\item 
for any $t>t_0$ and any interval $I \subset J \df \left[\frac{\pi}{2}-\frac{\varepsilon\pi}{4},\frac{\pi}{2}+\frac{\varepsilon\pi}{4}\right] \cup \left[\frac{3\pi}{2}-\frac{\varepsilon\pi}{4},\frac{3\pi}{2}+\frac{\varepsilon\pi}{4} \right] $, if  $q' \in \N$ and $\theta \in I$ satisfy $g_tr_{\theta}  q' \in \mathcal{W}$, then
\begin{align}\label{eq: small portion}
    |\{ \alpha \in I  :  g_tr_{\alpha}g_{-t} q' \notin \widehat{\mathcal{W}} \}| < \frac{\varepsilon}{8}|I|.
\end{align} 
\end{itemize}
Let
\begin{align*}
    \mathcal{V} \df \mathcal{W} \times \mathrm{int} (\B(\frac{\varepsilon}{8} R) \times \B(\frac{\varepsilon}{8} \delta)) \ \ \ \text{ and } \ \ \ \mathcal{U} = \bar \Phi(\mathcal{V}). 
\end{align*}

Since $\M$ has rank one, and since $\bar \Phi$ is a homeomorphism onto its image, $\mathcal{U}$ is open in $\M$, and contains $Z(K,r)$. Let $t>t_0$ and let
\begin{align*}
    I_{\mathcal{U},q} \df \{ \theta \in [0,2\pi] : g_tr_{\theta}q \in \mathcal{U}\}.
\end{align*}
We need to verify that $|I_{\mathcal{U},q}| <2\pi \varepsilon$. Since $|J|<\varepsilon \pi$, it is enough to show that $|I_{\mathcal{U},q} \cap J^c | <\pi \varepsilon$. For any $\theta \in I_{\mathcal{U},q}$, let $(q_{\theta}, x_{\theta},y_{\theta}) \in \mathcal{V}$ such that $g_tr_{\theta}q = \bar{\Phi}(q_{\theta}, x_{\theta},y_{\theta})$ and define 
\begin{align*}
    \mathcal{I}_{\theta} &\df \theta + \{ \alpha \in [0,2\pi]  :  g_t r_{\alpha} g_{-t}(x_{\theta} + \mathrm{i} y_{\theta}) \in B_{\N}^{\perp}(\frac{\varepsilon}{8}R,\frac{\varepsilon}{8}\delta) \}, \\
    \mathcal{J}_{\theta} &\df \theta + \{ \alpha \in [0,2\pi]  :  g_t r_{\alpha} g_{-t}(x_{\theta} + \mathrm{i} y_{\theta}) \in B_{\N}^{\perp}(R,\delta)\}.
\end{align*}

By definition, we know that $\|g_{-t}(x_{\theta} + \mathrm{i} y_{\theta})\|>\tilde{\delta}$. In particular, there is $\alpha \in [0,2\pi]$ such that $r_{\alpha} g_{-t}(x_{\theta} + \mathrm{i} y_{\theta}) = x + \mathrm{i} y$ with $\|y\|^2>\frac{\tilde{\delta}^2}{2}>\delta^2$ and thus $\alpha \notin \mathcal{J}_{\theta}$. Identifying $\C^n$ and $\mathfrak{R}^{\perp}_{\N}$ with a $\mathrm{GL}^+(2,\R)$-equivariant isometry, it follows from Proposition \ref{prop: ellipses and boxes} that 
\begin{align}\label{eq: small portion bis}
    | \mathcal{I}_{\theta} | < \frac{\varepsilon}{8} | \mathcal{J}_{\theta} |.
\end{align}

We claim that if $\alpha \in \mathcal{J}_{\theta} \smallsetminus \mathcal{I}_{\theta}$ then either $g_t r_{\alpha-\theta} g_{-t} q_{\theta} \notin \widehat{\mathcal{W}}$ or $g_t r_{\alpha} q \notin \mathcal{U}$. Indeed, let $\alpha \in \mathcal{J}_{\theta}$ and suppose $g_t r_{\alpha-\theta} g_{-t} q_{\theta} \in \widehat{\mathcal{W}}$. It follows from \eqref{eq: equivariance Rel} that
\begin{align*}
    g_tr_{\alpha}q &= (g_tr_{\alpha-\theta}g_{-t}) g_t r_{\theta}q \\
    &= (g_tr_{\alpha-\theta}g_{-t}) (q_{\theta} \pluscirc (x_{\theta} + \mathbf{i} y_{\theta})) \\
    &= g_tr_{\alpha-\theta}g_{-t} q_{\theta} \pluscirc g_tr_{\alpha-\theta}g_{-t}(x_{\theta} + \mathbf{i} y_{\theta}),
\end{align*}
\noindent and thus by injectivity of $\bar{\Phi}$, it follows that $g_tr_{\alpha}q \notin \mathcal{U}$, which proves our claim. 
Using \eqref{eq: small portion} and \eqref{eq: small portion bis}, 
we obtain: 
\begin{align*}
    |I_{\mathcal{U},q} \cap \mathcal{J}_{\theta} \cap J^c| &= |I_{\mathcal{U},q} \cap \mathcal{I}_{\theta} \cap J^c| +  |I_{\mathcal{U},q} \cap \mathcal{J}_{\theta} \smallsetminus \mathcal{I}_{\theta} \cap J^c | \\
    &\leq |\mathcal{I}_{\theta}| + \left| \left\{\alpha \in \mathcal{J}_{\theta} \cap J^c : g_t r_{\alpha - \theta} g_{-t} q_{\theta} \notin \widehat{\mathcal{W}} \right \} \right| \\ 
    &< \frac{\varepsilon}{8} | \mathcal{J}_{\theta} | + \frac{\varepsilon}{8} | \mathcal{J}_{\theta} \cap J^c | \leq  \frac{\varepsilon}{4} | \mathcal{J}_{\theta} |.
\end{align*}

 We have shown that the set $I_{\mathcal{U},q} \cap J^c$ is covered by the $(\mathcal{J}_{\theta})_{\theta \in I_{\mathcal{U},q}}$, and  from  this cover we can extract a sub-cover  $(\mathcal{J}_{\theta_l})_{l}$ such that any point is covered at most twice. This yields
%\begin{align*}
$$
    |I_{\mathcal{U},q} \cap J^c| \leq \sum_l |I_{\mathcal{U},q} \cap \mathcal{J}_l \cap J^c | 
    < \frac{\varepsilon}{2} \sum_l | \mathcal{J}_l | \leq   \varepsilon \pi.
$$
%\end{align*}
\end{proof}

\section{From measure classification to equidistribution of orbits}

For any rank-one invariant subvariety $\M$, we denote by $\mathcal{S}_{\M}$ the set of surfaces $q \in \M$ without horizontal saddle connections for which there is a proper invariant subvariety $\N \varsubsetneq \M$ and $x \in Z_{\N}^{\perp}$ such that $\Rel_x(q) \in \N$. By Proposition 1.3 in \cite{eskin2018invariant}, there are at most countably many invariant subvarieties of $\M$. We assume there are countably many, and denote them by $\left(\N^{(i)} \right)_{i \in \mathbb{N}}$ (the reader will have no difficulty in adapting our arguments to the case in which there are finitely many). With this notation, 
\begin{align*}
    \mathcal{S}_{\M} = \underset{(i,j) \in \mathbb{N} \times \mathbb{N}}{\bigcup} Z\left(\N^{(i)},j \right);
\end{align*}
\noindent in particular, the set $\mathcal{S}_{\M}$ is measurable.

The following result is independent of the weak classification of $U$-invariant measure and holds in arbitrary rank-one invariant subvarieties. It contrasts with a surprising phenomenon in $\mathcal{H}(1,1)$ discovered in \cite[Theorem 1.2]{chaika2020tremors}. 

\begin{thm}\label{preequidistribution}
Let $\M$ be a rank-one invariant subvariety, let $q \in \M_\infty \sm\mathcal{S}_{\M}$ and 

$$
\mu_{T} \df \frac{1}{T} \int_{0}^T u_{s\ast}\delta_q \,  ds. 
$$
Then any weak-$\ast$ limit of $(\mu_T)_{T>0}$ is a $U$-invariant Radon probability measure that assigns zero measure to $\M \sm \M_\infty$ and to $\mathcal{S}_{\M}$. 
\end{thm}

\begin{proof} 
Let $\nu$ be a weak-$\ast$ limit of a convergent subsequence $(\mu_{T_n})_{n>0}$. 

\textbf{Step 1.} We show that $\nu(\M)=1$. It is enough to show that for any $\varepsilon>0$, there is a compact $K \subset \M$ such that $\nu(K) \geq 1-\varepsilon$. We take $K$ as in item (II) in Theorem \ref{quantitative} for $r=1$. By assumption $q$ does not have any horizontal saddle connections and in particular, no horizontal saddle connection shorter than $r$. Thus there is  $T_0>0$ such that for any $T>T_0$, 
\begin{align*}
    \frac{1}{T} |\{s \in [0,T]  :  u_s  q \notin K \}| < \varepsilon. 
\end{align*}
It follows that for $n$ for which $T_n > T_0$, we have $\mu_{T_n}(K) \geq 1-\varepsilon$. Since $K$ is compact, by a well-known property of weak-* convergence, we obtain $\nu(K) \geq 1-\varepsilon$. 

\textbf{Step 2.} We now show that $\nu(\M\sm \M_{\infty})=0$. For any $i>0$, let $K_i \subset \M_\infty$ be the set of surfaces whose saddle connections have length  bounded below by $\frac{1}{i}$, and which have horizontal saddle connections, all of whose lengths are  bounded above by $i$. By construction the sets $(K_i)_{i>0}$ exhaust $\M\sm\M_{\infty}$ and thus it is enough to show that for any $i>0$ and any $\varepsilon>0$, $\nu(K_i)<\varepsilon.$ By a well-known property of weak-* convergence, it suffices to show that for each $\varepsilon>0$ and each $i$, there is an open set $\mathcal{U}$ that contains $K_i$, such that $\mu_T(\mathcal{U})<\varepsilon$ for any $T$ large enough. Once again, let $K$ be as in item (II) in Theorem \ref{quantitative} for $r=1$. Since the lengths of all saddle connection are bounded below on the compact set $K$, there is $t>0$ such that $g_{-t}(K_i) \subset K^c$. We denote by $\mathcal{U}$ the set $g_t(K^c)$ and let $q_t = g_{-t}  q$. Since $q_t$ does not have horizontal saddle connection either, there is a $T_0>0$ such that for any $T>T_0$,
\begin{align*}
    \frac{1}{T} |\{s \in [0,T] : u_s  q_t \notin K \}| < \varepsilon. 
\end{align*}

Now, let $T_1$ be large enough so that $e^{-2t}T_1 > T_0$ and let $T>T_1$. Then, using the commutation relation $g_{-t}u_sq = u_{s'} g_tq$, where $s' \df e^{-2t}s$, we obtain
\begin{align*}
    \mu_T(\mathcal{U}) &= \frac{1}{T} \int_0^T  \mathbf{1}_{\mathcal{U}}(u_s  q) \, ds \\ 
    &= \frac{1}{T} \int_0^T  \mathbf{1}_{K^c}(g_{-t} u_s  q) \, ds \\
    &= \frac{1}{Te^{-2t}} \left|\left\{s' \in \left[0,Te^{-2t}\right] :  u_{s'}  q_t \notin K  \right\} \right| < \varepsilon.
\end{align*}

\textbf{Step 3.} We now show that for any invariant subvariety $\N \varsubsetneq \M$ and $r>0$ we have $\nu(Z(\N,r))=0$. By Step 2, it is enough to show that for any $\varepsilon>0$ and any compact $K \subset \N_{\infty}$, there is an open set $\mathcal{U}$ that contains $Z(K,r)$ such that $\mu_T(\mathcal{U})<\varepsilon$ for any $T$ large enough. This follows from  Proposition \ref{frequencyofpassage}, applied to $L = \{q\}$.
\end{proof}

\begin{thm}[Genericity]\label{genericity}
Suppose $\M$ satisfies the weak classification of $U$-invariant measures. The orbit of any surface $q \in \M$ without horizontal saddle connections is generic for some measure as in \S \ref{push}. That is, for any $q \in \M_\infty$, there are $x \in Z_{\M}$, an invariant subvariety $\N \subset \M$ and $q' \in \N_\infty$ such that $q = \Rel_{x}(q')$ and
$$
\frac{1}{T} \int_0^T u_{s\ast} \delta_q \, ds \underset{T \to \infty}{\longrightarrow} \Rel_{x\ast} m_{\N} \ \ \ \ \text{ (w.r.t. the weak-* topology).}
$$

\end{thm}

\begin{remark}
It is instructive to compare Theorem \ref{genericity} with \cite[Thm. 1.4]{chaika2020tremors}, which gives examples of surfaces in the stratum $\H(1,1)$ which are not generic for any measure. Note that $\H(1,1)$ is not of rank one. It should also be noted that the genericity result in \cite[Thm. 11.1]{bainbridge2016horocycle} is stronger than Theorem \ref{genericity}, as it does not assume that $q$ has no horizontal saddle connections. This is possible since it relies on a stronger measure classification result. 
\end{remark}

\begin{proof}
Let $\N$ be the invariant subvariety of smallest dimension amongst the collection of all invariant subvarieties $\mathcal{O}$ such that there is $\xi \in Z_{\M}$ satisfying $\Rel_{\xi}(q) \in \mathcal{O}$. Such a pair $(\N,\xi)$ is unique by Proposition \ref{injectivity}. For any $T>0$, we define
\begin{align*}
    \nu_T \df\frac{1}{T} \int_0^T u_{s\ast} \delta_{\Rel_{\xi}(q)} \, ds.
\end{align*}
\noindent The measure $\nu_T$ is supported on $\N$. By definition of $\N$, we have that $\Rel_{\xi}(q) \notin \mathcal{S}_{\N}$ and since $q$ does not have horizontal saddle connections we have that $\Rel_{\xi}(q) \in \N_{\infty}$. It follows from Theorem \ref{preequidistribution} (applied to $\N$ and $\Rel_{\xi}(q)$) that any limit $\nu$ of a convergent subsequence of $\nu_T$ is a Radon probability measure that assigns zero measure to the set of surfaces with horizontal saddle connections, as well as to  the singular set $\mathcal{S}_{\N}$. We decompose the measure $\nu$ into $U$-ergodic components; that is, 
\begin{align*}
    \nu = \int_{\M} \nu_p \ d\nu(p)
\end{align*}
\noindent where there is a $\nu$-conull set $E$ comprised of surfaces $p \in \M$ for which $\nu_p$ is an ergodic $U$-invariant probability Radon measure that assigns zero measure to $\N \sm \N_\infty$ and to $\mathcal{S}_{\N}$. Let $p \in E$. Since $\M$ satisfies the weak classification of $U$-invariant measures, $\nu_p$ is a measure obtained by pushing a measure $m_{\N'}$ on an invariant subvariety $\N'$ by an element of real Rel. Since $\nu_p$ is supported in $\N$ and $\nu_p(\mathcal{S}_{\N})=0$, it must be that $\nu_p = m_{\N}$. Therefore also $\nu = m_{\N}$, and since $\nu$ was an arbitrary subsequential limit of $\nu_T$, we have that $\nu_T \longrightarrow_{T \to \infty} m_\N$ (with respect to the weak-* topology).

This amounts to saying that $q' \df \Rel_{\xi}(q)$ is generic for $m_{\N}$. From Proposition \ref{prop: first} we obtain that $q = \Rel_{-\xi}(q')$ is generic for $\Rel_{-\xi \ast}m_{\N}$, which is exactly what we wanted with $x = -\xi$. 
\end{proof}

The next result shows that the convergence in Theorem \ref{genericity} holds uniformly on compact sets, outside finitely many obvious exceptions. The statement is modeled on \cite[Thm. 3]{DaniMargulis} and will be useful in \cite{CWY2}. It is new even in the case of  eigenform loci in $\H(1,1)$ studied in \cite{bainbridge2016horocycle}. 

\begin{thm}[Uniform equidistribution]\label{uniform}
Suppose $\M$ satisfies the weak classification of $U$-invariant measures, let $K \subset \M$ be compact, let $\varepsilon >0$ and let $f \in C_c(\M)$. There are finitely many invariant subvarieties $\N^{(1)} , \ldots, \N^{(k)}$ contained in $\M$ and $r>0$ such that for any compact $L \subset K \cap \M_r \sm \bigcup_{i=1}^k Z\left(\N^{(i)},r \right)$, there is $T_0>0$ such that for all $ T> T_0$ and all $q \in L$, we have
\begin{align*}
    \left  | \frac{1}{T} \int_0^T f(u_s  q) \, ds - \int_{\M} f \ dm_{\M} \right | < \varepsilon.
\end{align*}
\end{thm}

\begin{proof}
Let $\left(\N^{(i)} \right)_{i\in \mathbb{N}^{\ast}}$ be the collection of all invariant subvarieties properly contained in $\M$ (if the collection is finite we index it so that $\N^{(i)}$ is eventually constant). Suppose by way of contradiction that there are a compact set $K \subset \M$, a function $f \in C_{c}(\M)$ and $\varepsilon_0>0$  such that the following  holds:

\medskip 

\textbf{Contradicting assumption: } {\em For all $n \geq 1$ and $r>0$, there is a compact $L \subset K \cap \M_r \sm \bigcup_{i =1}^n Z\left(\N^{(i)},r\right)$ such that for all $S>0$, there is $T>S$ and $q \in L$ such that:
\begin{align}\label{eq: epsilon}
     \left| \frac{1}{T} \int_0^T f(u_s  q) \ ds - \int_{\M} f \ dm_{\M}  \right| \geq \varepsilon_0.
\end{align}
}

\medskip

For any $i \in \mathbb{N}$, let $(K_{i,j})_{j \in \mathbb{N}}$ be an exhaustion by compact subsets of $\N^{(i)}_\infty$ (the set of surfaces without horizontal saddle connections in $\N^{(i)}$), let 
\begin{align*}
    K_n \df \underset{i\leq n} \bigcup Z(K_{i,n},n), 
\end{align*} 
\noindent and let $L_n$ be the compact subset comprised of surfaces that have a horizontal saddle connection of length at most $n$ and whose shortest saddle connection is longer than $\frac{1}{n}$. We claim that there is a sequence of surfaces $(q_n)_{n \in \mathbb{N}}$ in $K$, a sequence of open sets $(\mathcal{U}_n)_{n \in \mathbb{N}}$, and a sequence of times $(T_n)_{n\in\mathbb{N}}$ with $T_n > n$ such that:

\begin{enumerate}
    \item[(i)] for any $n$, $K_n \cup L_n \subset \mathcal{U}_n$;
    \item[(ii)] for any $n$, 
        \begin{align}\label{eq: this will contradict}
            \left| \frac{1}{T_n} \int_0^{T_n} f(u_s  q_n) \, ds - \int_{\M} f \ dm_{\M} \right | \geq \varepsilon_0;
        \end{align}
    \item[(iii)] for any $n>0$, for all $0<i\leq n$,
        \begin{align}\label{littletime}
            \frac{1}{T_n} |\{s \in [0,T_n]  :  u_s q_n \in \mathcal{U}_i \}| \leq \frac{1}{i}. 
        \end{align}
\end{enumerate}

Indeed, for any $n \geq 1$ and any $0< i\leq n$, let $\mathcal{U}_{i,n}$ and $R_{i,n}>n$ be as in Proposition \ref{frequencyofpassage} for $K=K_{i,n}$, $r=n$ and $\varepsilon = \frac{1}{2n^2}$. Let also $\mathcal{W}_n$ and $R_n'$ be be the $\mathcal{U}$ and $R$ given by Proposition \ref{frequencyofpassagesaddleconnection} for $\varepsilon=\frac{1}{2n}$ and $r=n$. Now, define $R_n = \max(R_{1,n},\ldots,R_{n,n}, R_n')$ and $\mathcal{U}_{n} = \bigcup_{i \leq n} \mathcal{U}_{i,n} \cup \mathcal{W}_n$. By construction, $\mathcal{U}_n$ contains $K_n \cup L_n$ and satisfies the following property: for any compact subset $L$ of $\M_{R_n} \sm \bigcup_{i\leq n} Z \left(\N^{(i)},R_{n} \right)$, there is $S_n>0$ such that 
\begin{align}\label{not much}
   \forall q \in L, \  \forall T>S_n, \  \frac{1}{T} |\{t \in [0,T]  :  u_s q \in \mathcal{U}_{n} \}|
  \leq \frac{1}{n}. 
\end{align} 

Now, let $n \geq 1$, let $r_n = \max(R_1,\dots, R_{n})$ and let $L$ be as in the contradicting assumption for $n$ and $r_n$. For any $i \leq n$, the set $L$ is a compact subset of $\M_{R_i} \sm \bigcup_{j\leq n} Z\left(\N^{(j)},R_{i} \right)$ and let $S_i$ be as in \eqref{not much}. We define $S = \max(S_1,\dots,S_n,n)$. By the contradicting assumption, there is $q_n \in L$ together with $T_n>S$ such that \eqref{eq: this will contradict} is satisfied. By definition of $T_n$, the surface $q_n$ satisfies \eqref{littletime}. Since $r_n>n$, this concludes the construction of $(q_n)_{n \in \mathbb{N}}$. 

Now define a sequence of measures by 
\begin{align*}
    \mu_n \df \frac{1}{T_n} \int_0^{T_n} \delta_{u_s q_n} \, ds.
\end{align*}

Our goal is to show that this sequence of measures weak-* converges to $m_{\M}$, which will contradict \eqref{eq: this will contradict}. By passing to a subsequence, we can suppose that $(\mu_n)_{n>0}$ converges to some measure $\nu$, and we need to show $\nu = m_\M$. We follow the strategy used in the proof Theorem \ref{preequidistribution}.  

\textbf{Step 1}. We show that $\nu(\M)=1$. It is enough to show that for any $\varepsilon>0$, there is a compact $K_0 \subset \M$ such that $\mu_n(K_0)\geq 1-\varepsilon$ for any $n$ large enough. We can choose $K_0$ and $T_0$ as in item (I) in Theorem \ref{quantitative} for $r=1$ and $L=K$. For any $n>T_0$ we have $T_n>T_0$ and thus
\begin{align*}
    \mu_n(K_0) = \frac{1}{T_n} |\{ t \in [0,T_n]  : u_t  q_n \in K_0 \}| \geq 1-\varepsilon. 
\end{align*}

\textbf{Step 2.} We now want to show that $\nu(\M\sm\M_{\infty})=0$. By construction, the $(L_i)_{i>0}$ exhaust $\M\sm\M_{\infty}$ and thus it is enough to show that for any $\varepsilon>0$ and $i>0$, there is an open set $\mathcal{U}$ that contains $L_i$ and such that $\mu_n(\mathcal{U})<\varepsilon$ for any $n$ large enough. Let
$$
N > \max\left(i,\frac{1}{\varepsilon}\right)
$$ 
\noindent and set $\mathcal{U} = \mathcal{U}_N$. By (i) and the definition of $N$, $\mathcal{U}_N$ contains $L_i$ and for any $n>N$ we have from \eqref{littletime} that
\begin{align*}
    \mu_n(\mathcal{U}) = \frac{1}{T_n} |\{t \in [0,T_n] : u_s q_n \in
    \mathcal{U}_N  \}| \leq  \frac{1}{N} < \varepsilon. 
\end{align*} 

\textbf{Step 3.} We now prove that if $\N \varsubsetneq \M$ is an invariant subvariety and $r>0$, then $\nu(Z(\N,r))=0$. Let $i>0$ such that $\N = \N^{(i)}$. We recall that we had an exhaustion by nested compact sets $(K_{i,j})_{j>0}$ of $\N^{(i)}_{\infty}$. It is enough to show that for any $\varepsilon>0$ and $j>0$, there is a open set $\mathcal{U}$ that contains $Z(K_{i,j},r)$ and such that $\mu_n(\mathcal{U})<\varepsilon$ for any $n$ large enough. Let 
$$
N > \max\left(i,j,r, \frac{1}{\varepsilon}\right)
$$
and set $\mathcal{U} = \mathcal{U}_N$. By (i) and the definition of $N$, $\mathcal{U}_N$ contains $Z(K_{i,j},r)$ and for any $n>N$ we have from \eqref{littletime} that
\begin{align*}
    \mu_n(\mathcal{U}) = \frac{1}{T_n} |\{t \in [0,T_n] : u_s q_n \in \mathcal{U}_N  \}| \leq  \frac{1}{N+1} < \varepsilon. 
\end{align*} 

\textbf{Step 4.} Since for any $n>0$, we have $T_n > n$, the measure $\nu$ is a $U$-invariant Radon measure. It follows from Step 1 that it is a probability measure, from Step 2 that it assigns measure zero to the set of surfaces that have a horizontal saddle connection and from Step 3 that it assigns measure zero to the set $\mathcal{S}_{\M}$. Let
\begin{align*}
    \nu = \int_{\M} \nu_p \ d\nu(p)
\end{align*}
\noindent be the decomposition of $\nu$ into $U$-invariant ergodic measures. There is a $\nu$-conull set $E$ of surfaces $p \in \M$ for which $\nu_p$ is an ergodic $U$-invariant probability Radon measure that assigns measure zero to the set of surfaces with horizontal saddle connections and such that $\nu_p(\mathcal{S}_{\M})=0$. By the weak classification of $U$-invariant measure, for every $p \in E$ we have $\nu_p= m_{\M}$, and hence $\nu = m_{\M}$. This is a contradiction to \eqref{eq: this will contradict}, concluding the proof.
\end{proof}

\section{Equidistribution results for sequences of measures}
In this section we state and prove several equidistribution results regarding limits of sequences of measures. The first result extends \cite[Thm. 1.5]{bainbridge2016horocycle}. 

\begin{thm}\label{pushbyrel} 
Suppose $\M$ satisfies the weak classification of $U$-invariant measures, let $\mathcal{L} \subset \M$ be an invariant subvariety and let $x \in Z_{\M}$. There are an invariant subvariety $\N \subset \M$ that contains $\mathcal{L}$ and $x \in Z_{\N}$ such that, in the weak-* topology, ${\Rel_{tx}}_{\ast}m_{\mathcal{L}} \to m_{\N}$ as $t \to \infty$. 
\end{thm}

For the proof of Theorem \ref{pushbyrel} we will need the following:

\begin{lem}\label{injectivitybis}
Suppose $\N_1, \N_2 \varsubsetneq \M$ are two invariant subvarieties. If there is $r>0$ such that $m_{\N_1}(Z(\N_2,r))>0$ then $\N_1 \subset \N_2$. 
\end{lem}

\begin{proof}
Since $m_{\N_1}$ is $U$-invariant and ergodic, we have that $m_{\N_1}(Z(\N_2,r))=1$. For any $t$, since $m_{\N_1}$  is $g_t$-invariant and $g_{-t}(Z(\N_2,r)) = Z(\N_2,e^{-t}r)$, we have $m_{\N_1}(Z(\N_2,e^{-t}r))=1.$ It follows that $(\N_2)_\infty = \bigcap_{n >0} Z(\N_2,re^{-n})$ also has full measure for $m_{\N_1}$. Since $m_{\N_1}$ assigns zero measure to surfaces with horizontal saddle  connections, we deduce that $m_{\N_1}(\N_2) = 1$. This implies that $\N_2$ contains $\N_1 = \mathrm{supp}\, m_{\N_1}$. 
\end{proof}

\begin{proof}[Proof of Theorem \ref{pushbyrel}]
Let $\mathcal{C}(\mathcal{L},x)$ be the collection of invariant subvarieties $\N$ of $\M$ that contain $\mathcal{L}$ and such that $x \in Z_{\N}$, and let $\N \in \mathcal{C}(\mathcal{L},x)$ be of smallest dimension.
We will show that $\mu_t= {\Rel_{t x}}_{\ast}m_{\mathcal{L}}$ converges to $m_\N$ as $t \to \infty$. 

It is enough to show that if  $\mu_n \df \mu_{t_n}$ is a convergent subsequence then its limit is $m_{\N}$. Denote by $\nu$ the limit of $\mu_n$. By construction, $\mathrm{supp} \, \nu \subset  \N$.  

\textbf{Step 1.} We show that $\nu$ is a probability measure. It is enough to show that for any $\varepsilon>0$ there is compact $K \subset \M$ such that $\mu_n(K) \geq 1-\varepsilon$ for $n$ large
enough. We choose $K$ as in Theorem \ref{quantitative}, item (II) for $r=1$. The measure $m_{\mathcal{L}}$ assigns full measure to surfaces without horizontal saddle connections, and since this set is invariant under $\Rel_{tx}$ we have $\mu_n(\M_{\infty})=1$. Therefore there is $q_n \in \M_\infty$ that is Birkhoff generic for the indicator function $\mathbf{1}_K$ with respect to $\mu_n$. By the choice of $K$, for any $T$ large enough, 
\begin{align*}
    \frac{1}{T} |\{s \in [0,T] : u_s q_n \in K \}| > 1- \varepsilon.
\end{align*}
By Birkhoff genericity, the left-hand side tends to $\mu_n(K)$ as $T \to \infty$ and thus $\mu_n(K) \geq 1-\varepsilon$.   

\textbf{Step 2.} We now show that $\nu(\M\sm \M_{\infty})=0$. For any $i>0$, let $K_i$ be the closure of the set of surfaces whose saddle connection lengths are bounded below by $\frac{1}{i}$, and which have at least one horizontal saddle connection, with the lengths of all horizontal saddle connections bounded above by $i$. By construction the $(K_i)_{i>0}$ exhaust $\M\sm\M_{\infty}$ and thus it is enough to show that for any $i>0$ and any $\varepsilon>0$, there is an open set $\mathcal{U}$ that contains $K_i$ and such that $\mu_n(\mathcal{U}) \leq \varepsilon$ for any $n$ large enough. Once again, let $K$ be as in item (II) in Theorem \ref{quantitative} for $r$ as above. Since the length of saddle connections is bounded below on compact sets and $K$ is compact, there is a $t>0$ such that $g_{-t}(K_i) \subset K^c$, so that
$$
K_i \subset \mathcal{U} \df g_t(K^c).
$$ 

Let $n$ be large enough so that $t_n>t$. We have $\mu_n(\M \sm \M_{re^{t_n}})=0$ and thus there is a surface $q_n$ without horizontal saddle connections shorter than $re^{t_n}$ that is Birkhoff generic for $\mathbf{1}_{K^{c}}$ with respect to $\mu_n$. The surfaces $q_n' \df  g_{-t}  q_n$ do not have horizontal saddle connections shorter that $r$ and thus there is $T_0>0$ such that for any $T>T_0$,
\begin{align*}
    \frac{1}{T} |\{s \in [0,T]  : u_s  q_n' \notin K \}| < \varepsilon. 
\end{align*}

Now, let $T_1$ be large enough so that $e^{-2t}T_1 > T_0$. For any $T>T_1$ we have
\begin{align*}
    &  \frac{1}{T} \int_0^T  \mathbf{1}_{\mathcal{U}}(u_s  q_n) \, ds = \frac{1}{T} \int_0^T \mathbf{1}_{K^c}(g_{-t} u_s q_n) \ ds \\ 
    =& \frac{1}{Te^{-2t}} \left|\left\{s \in \left[0,Te^{-2t} \right] : u_s  q_n' \notin K \right\} \right| < \varepsilon.
\end{align*}
By Birkhoff genericity the quantity on the left-hand side converges to $\mu_n(\mathcal{U})$  as $T \to \infty$ and thus $\mu_n(\mathcal{U})\leq \varepsilon$. 

\textbf{Step 3.} Let $\N' \varsubsetneq \N$ be an invariant subvariety and let $r>0$. We want to show that $\nu(Z(\N',r)=0$. It is enough to show that for any compact subset $K \subset \N'_{\infty}$ and $\varepsilon>0$ there is an open set $\mathcal{U}$ that contains $Z(K,r)$ such that $\mu_n(\mathcal{U}) \leq \varepsilon$ for $n$ large enough. Let $\mathcal{U}$ and $R>r$ be as in Proposition \ref{frequencyofpassage}.

Suppose first that $\mathcal{L}$ is not contained in $\N'$. If $\mu_n(Z(\N',R))>0$ then $m_{\mathcal{L}}(Z(\N',R+t_n\|x\|))>0$ and thus by Lemma \ref{injectivitybis}, we get that $\mathcal{L} \subset \N'$, which is a contradiction. It follows that $\mu_n(Z(\N',R))=0$ and thus by Birkhoff's Theorem, one can find a point $q_n \in \M \sm Z(\N',R)$ such that
\begin{align*}
    \mu_n(\mathcal{U}) = \underset{T \to \infty}{\lim} \frac{1}{T} |
    \{s \in [0,T]  :  u_s q_n \in \mathcal{U} \} |. 
\end{align*}

Applying Proposition \ref{frequencyofpassage} to $L=\{q_n\}$ shows that for all $T$ large enough,
\begin{align*}
    \frac{1}{T} | \{s\in [0,T]  :  u_s q_n \in \mathcal{U} \} | < \varepsilon.
\end{align*}
This shows that $\mu_n(\mathcal{U}) \leq \varepsilon$.

Suppose now that $\mathcal{L}\subset \N'$. By definition of $\N$, we know that $x \notin Z_{\N'}$. It thus follows from Proposition \ref{injectivity} that the map
\begin{align*}
    \Phi: \N_{\infty}' \times \mathbb{R} \to \N, \ \ \ \ \ \Phi(q,x) \df \Rel_{t x}(q)
\end{align*} 
\noindent is injective. This means that for any $t_n$ large enough such that $t_n \|x\|>R$, the surface $q_n = \Rel_{t_n x}(q_n')$ does not belong to $Z(\N,R)$ and thus applying once more Proposition \ref{frequencyofpassage} to $L=\{q_n\}$ shows that $\mu_n(\mathcal{U})\leq \varepsilon$ for $n$ large enough.

To deduce now that $\nu = m_{\N}$, we use the ergodic decomposition argument which we used in Step 4 of the Proof of Theorem \ref{uniform}. 
\end{proof}

\begin{thm}\label{geodesicpush}
Suppose $\M$ satisfies the weak classification of $U$-invariant measures and let $\mu$ be a $U$-invariant  ergodic Radon probability measure on $\M$. Then there is an invariant subvariety $\N \subset \M$ such that  ${g_t}_{\ast}\mu$ converges to $m_{\N}$ as $t\to \infty$. If we furthermore assume that $\mu(\M_{\infty}) = 1$, then there is an invariant subvariety $\mathcal{L} \subset \N$ such that ${g_t}_{\ast}\mu$ converges to $m_{\mathcal{L}}$ as $t \to -\infty$. 
\end{thm}

\textbf{Remark}. If $\mu$ is an ergodic $U$-invariant measure with $\mu(\M_{\infty})<1$ then by ergodicity $\mu$-a.e $q \in \M$ has a horizontal saddle connection and ${g_t}_{\ast}\mu$ converges to the zero measure as $t \to -\infty$.

\begin{proof}[Proof of Theorem \ref{geodesicpush}]
We first prove the first assertion. We assume first that for any invariant subvariety $\N \varsubsetneq \M$, we have $\mu(\N)=0$. In that case, we are going to show that ${g_t}_{\ast}\mu$ converges to $m_{\M}$ as $t \to \infty$. Let $\mu_n = {g_{t_n}}_{\ast}\mu$ be a convergent subsequence when $t \to \infty$ and let $\nu$ be its limit. If generic surfaces for $\mu$ have some horizontal saddle connections, let $r$ be the length of the shortest of them. Otherwise, let $r=1$. 

\textbf{Step 1.} We show that $\nu(\M)=1$. Let $\varepsilon >0$. It is enough to show that there is a compact $K$ such that $\mu_n(K)\geq1-\varepsilon$. Let $K$ be a compact
as in item (II) of Theorem \ref{quantitative} with $r, \varepsilon$ as above. The measure $\mu_n$ assigns measure zero  to  surfaces with a saddle connection shorter than $r$ and thus there is a surface $q_n$ that is Birkhoff generic for the indicator function $\mathbf{1}_K$ with respect to $\mu_n$. By  definition of $K$, for any $T$ large enough,
\begin{align*}
    \frac{1}{T} |\{s \in [0,T] : u_s  q_n \in K \}| \geq1-\varepsilon.
\end{align*}
By Birkhoff genericity, the quantity on the left-hand side converges to $\mu_n(K)$ as $T \to  \infty$ and thus $\mu_n(K)\geq1-\varepsilon$. 

\textbf{Step 2.} We now show that $\nu(\M\sm \M_{\infty})=0$. For any $i>0$, let $K_i$ be the closure of the set  of surfaces whose saddle connection lengths are bounded below by $\frac{1}{i}$,  and which have at least one horizontal saddle connection, with the lengths of all horizontal saddle connections bounded above by $i$. By construction the $(K_i)_{i>0}$ exhaust $\M\sm\M_{\infty}$ and thus it is enough to show that for any $i>0$ and any $\varepsilon>0$, there is an open set $\mathcal{U}$ that contains $K_i$ and such that $\mu_n(\mathcal{U}) \leq \varepsilon$ for any $n$ large enough. Once again, let $K$ be as in item (II) in Theorem \ref{quantitative} for $r$ as above. Since the length of saddle connections is bounded below on compact sets and $K$ is compact, there is a $\tau>0$ such that $g_{\tau}(K_i) \subset K^c$, so that

$$
K_i \subset \mathcal{U} \df g_{\tau}(K^c).
$$ 

Let $n$ be large enough so that $t_n>\tau$. We have $\mu_n(\M_{re^{t_n}})=0$ and thus there is a surface $q_n$ without horizontal saddle connections shorter than $re^{t_n}$ that is Birkhoff generic for $\mathbf{1}_{K^{c}}$ with respect to $\mu_n$. The surfaces $q_n' \df  g_{-\tau} q_n$ do not have horizontal saddle connections shorter that $r$ and thus there is $T_0>0$ such that for any $T>T_0$,
\begin{align*}
    \frac{1}{T} |\{s \in [0,T]  : u_s  q_n' \notin K \}| < \varepsilon. 
\end{align*}

Now let $T_1$ be large enough so that $e^{-2t}T_1 > T_0$. For any $T>T_1$ we have
\begin{align*}
    &  \frac{1}{T} \int_0^T  \mathbf{1}_{\mathcal{U}}(u_s  q_n) \, ds = \frac{1}{T} \int_0^T \mathbf{1}_{K^c}(g_{-\tau} u_s q_n) \ ds \\ 
    =& \frac{1}{Te^{-2\tau}} \left|\left\{s \in \left[0,Te^{-2\tau} \right] : u_s  q_n' \notin K \right\} \right| < \varepsilon.
\end{align*}

By Birkhoff genericity the quantity on the left-hand side converges to $\mu_n(\mathcal{U})$  as $T \to \infty$ and thus $\mu_n(\mathcal{U})\leq \varepsilon$. 

\textbf{Step 3.} We now show that for any invariant subvariety $\N \varsubsetneq \M$ and $r>0$ we have $\nu(Z(\N,r))=0$. It is enough to show that for any compact subset $K$ of $\N_{\infty}$ and $\varepsilon>0$, there is a open set $\mathcal{U}$ that contains $Z(K,r)$ and such that $\mu_n(\mathcal{U})<\varepsilon$ for any $n$ large enough. We claim that there is $r_0>0$ such that $\mu(Z(\N,r_0))=0$. Indeed, suppose that for any $n>1$, we have $\mu(Z(\N,\frac{1}{n}))>0$. Since $\mu$ is ergodic and the sets $Z(\N, r)$ are $U$-invariant, we actually have $\mu \left(Z(\N,\frac{1}{n}) \right)=1$ and thus $\N_\infty = \bigcap_{n>1}Z(\N,\frac{1}{n})$ also has full measure, contradicting the assumption on $\mu$ made at the beginning of the proof. This proves the claim.

According to Proposition \ref{frequencyofpassage}, there is  $R>r$ and a neighborhood $\mathcal{U}$ such that for any compact subset $L$ of $\M \sm Z(\N,R)$, there is $T_0> 0$ such that for any $T>T_0$, 
\begin{align*}
    \frac{1}{T} \ |\{s \in [0,T]  : u_s q \in \mathcal{U} \}| < \varepsilon     .
\end{align*}

Let $n$ be large enough so that $e^{-t_n}R<r_0$ and let $q_n$ be a generic point for the measure $\mu_n$. If $q_n \in Z(\N,R)$ then the measure $\mu$ is supported on $Z(\N,e^{-t_n} R)$ and thus $\mu(Z(\N,r_0))=1$, a contradiction. Thus we can apply Proposition \ref{frequencyofpassage} to $L=\{q_n\}$ to find $T_0$ such that for any $T>T_0$, we have
\begin{align*}
    \frac{1}{T} \ |\{s \in [0,T]  : u_s  q_n \in \mathcal{U} \}| < \varepsilon  .    
\end{align*}

By the Birkhoff Ergodic Theorem, the left-hand side converges to $\mu_n(\mathcal{U})$ when $T$ tends to $+\infty$ and thus $\mu_n(\mathcal{U})<\varepsilon$ for any $n$ large enough. We prove as in the final step of the proof of Theorem \ref{uniform} that $\nu = m_{\M}$.  

\textbf{General case}. Suppose now $\mu$ is any $U$-invariant ergodic  Radon probability measure and consider the collection $\mathcal{C}(\mu)$ of all the invariant subvarieties that have
positive measure for $\mu$. Then in fact, by ergodicity, $\mu(\N)=1$ for any $\N \in \mathcal{C}(\mu)$. Let $\N_0 \in \mathcal{C}(\mu)$ of smallest dimension. We claim that there is no $\N_1 \in \mathcal{C}(\mu)$ which is properly contained in $\N_0$. Indeed, if $\N_1 \in \mathcal{C}(\mu)$ with $\N_1 \subset \N_0$ then $\dim \N_1 =
\dim \N_0$ by choice of $\N_0$, so that $\N_1$ is open in $\N_0$. This implies that $m_{\N_0}(\N_1)>0$,  which implies by ergodicity and $G$-invariance that  $m_{\N_0} = m_{\N_1}$ and hence $\N_0 = \N_1$. Thus we can now apply the first part of the proof to $\N_0$ in place of $\M$, which proves that ${g_t}_{\ast}\mu$ converges to $m_{\N_0}$.   

\medskip

We now prove the second assertion of the statement. We assume that $\mu(\M_{\infty})=1$. Since $\M$ satisfies the weak classification of $U$-invariant measures, it follows that there is an invariant subvariety $\mathcal{L} \subset \M$ and $x \in Z_{\M}$ such that $\mu = {\Rel_{x}}_{\ast}m_{\mathcal{L}}$. Notice that for any $t \in \R$, we have ${g_t}_{\ast}\mu = {\Rel_{e^tx}}_{\ast}m_{\mathcal{L}}$. It follows from Theorem \ref{pushbyrel} and the first part of the proof that $\mathcal{L} \subset \N$. Finally, it is clear that  ${g_t}_{\ast}\mu$ converges to $m_{\mathcal{L}}$ when $t \to -\infty$. 
\end{proof}

\begin{thm}\label{thm: circle averages}
Suppose $\M$ satisfies the weak classification of horocycle invariant measures. Let $q \in \M$, $\N \df \overline{Gq} \subset \M$, 
\begin{align*}
    \mu_q \df \frac{1}{2\pi} \int_0^{2\pi}  {r_{\theta}}_{\ast} \delta_q \ d\theta,
\end{align*}
and let $\mu_{t,q}$ denote the `pushed circle average' ${g_t}_{\ast} \mu$. Then $\mu_{t,q} \underset{t \to \infty}{\longrightarrow} m_{\N}.$
\end{thm}

\begin{proof}
Up to replacing $\M$ by $\overline{Gq},$
we may assume that $q$ is not contained in any invariant subvariety $\N \subsetneq \M$. Let $\nu$ be a weak-$\ast$ limit of a convergent subsequence $(\mu_{t_n,q})_{n>0}$. We want to show that $\nu = m_{\M}$.

\textbf{Step 1.} We want to show that $\nu(\M)=1$. It is enough to show that for any $\varepsilon>0$, there is a compact $K \subset \M$ such that $\nu(K)>1-\varepsilon$. This follows from \cite[Theorem 5.2]{eskin2001asymptotic}. 

\medskip

\textbf{Step 2.} 
We now want to show that $\nu(\M\sm\M_{\infty})=0$. By construction, the $(L_i)_{i>0}$ exhaust $\M\sm\M_{\infty}$ and thus it is enough to show that for any $\varepsilon>0$ and $i>0$, there is a open set $\mathcal{U}$ that contains $L_i$ and such that $\mu_{t,q}(\mathcal{U})<\varepsilon$ for any $t$ large enough. It follows once again from \cite[Theorem 5.2]{eskin2001asymptotic} that there are $t_1>0$ and $\delta>0$ such that for any $t > t_1$, we have $\mu_{t,q}(C_{\delta})<\varepsilon$, where $C_{\delta}$ is the open set of surfaces that have a saddle connection strictly shorter than $\delta$. Let $t_2 = \log(\frac{\delta}{2}) - \log(i)$ and notice that  $g_{t_2} L_i \subset C(\delta)$. Let $\mathcal{U} = g_{-t_2}(C(\delta))$. For any $t>t_1 - t_2$, we have
\begin{align*}
    \mu_{t,q}(\mathcal{U})&= \mu_{t+t_2,q}(C(\delta)) < \varepsilon.
\end{align*}

\textbf{Step 3.} We now want to show that for any invariant subvariety $\N \varsubsetneq \M$ and any $r>0$, we have $\nu(Z(\N,r))$ = 0. It is enough to
show that for any compact subset $K$ of $\N_{\infty}$ and
$\varepsilon>0$, there is a open set $\mathcal{U}$ that contains
$Z(K,r)$ and such that $\mu_{t,q}(\mathcal{U})<\varepsilon$ for any $t$
large enough. This is exactly Proposition \ref{prop: push of circles}.

\medskip

It follows that $\nu$ is a $U$-invariant probability measure that assigns measure zero to the set of surfaces with horizontal saddle connections as well as to any Rel push of any invariant subvariety $\N \subsetneq \M$. It follows by ergodic decomposition, and using the weak classification of $U$-invariant measures, that $\nu = \mu_{\M}$. 
\end{proof}

By a well-known argument developed in  \cite{eskin2001asymptotic}, this implies 

\begin{thm} 
Suppose $\M$ satisfies the weak classification of horocycle invariant measures. Let $q \in \M$ and for any $L>0$, denote by $N_L(q)$ the number of saddle connection of length at most $L$ on $q$. Then there is a constant $C>0$ such that 
\begin{align*}
    \frac{1}{L^2} N_L(q) \underset{L \to \infty}{\longrightarrow} C.
\end{align*}
\end{thm}

\bibliographystyle{alpha}
\bibliography{Rel_weak_classification.bib}

\end{document}